\documentclass[a4paper,reqno,10pt]{amsart}
\usepackage{amsfonts,amsmath,amsthm,amssymb,stmaryrd,mathrsfs}
\newtheorem{theorem}{Theorem}[section]
\newtheorem{lemma}[theorem]{Lemma}

\newtheorem{Remark}[theorem]{Remark}

\newtheorem{Corollary}[theorem]{Corollary}
\numberwithin{equation}{section}
\allowdisplaybreaks

\usepackage[left=1 in, right=1 in,top=1 in, bottom=1 in]{geometry}

\newcommand{\na}{\nabla}

\newcommand{\al}{\alpha}

\newcommand{\la}{\lambda}

\providecommand{\norm}[1]{\left\Vert#1\right\Vert}

\providecommand{\norms}[1]{\left\vert#1\right\vert}
\def\r3{\mathbb{R}^3}

\begin{document}
\title[Compressible bipolar Euler-Maxwell system]{Global well-posedness of the compressible bipolar Euler-Maxwell system in $\r3$}

\author{Zhong Tan}
\address{School of Mathematical Sciences\\
Xiamen University\\
Xiamen, Fujian 361005, China}
\email[Z. Tan]{ztan85@163.com}

\author{Yong Wang}
\address{School of Mathematical Sciences\\
Xiamen University\\
Xiamen, Fujian 361005, China}
\email[Y. Wang]{wangyongxmu@163.com}



\thanks{Corresponding author: Yong Wang, wangyongxmu@163.com}

\begin{abstract}
We first construct the global unique solution by assuming that the initial data is small in the $H^3$ norm but its higher order derivatives could be large. If further the initial data belongs to $\Dot{H}^{-s}$ ($0\le s<3/2$) or $\dot{B}_{2,\infty}^{-s}$ ($0< s\le3/2$), we obtain the various decay rates of the solution and its higher order derivatives. As an immediate byproduct, the  $L^p$--$L^2$ $(1\le p\le 2)$ type of the decay rates follow without requiring the smallness for $L^p$ norm of initial data. In particular, the decay rate for the difference of densities could reach to $(1+t)^{-\frac{13}{4}}$ in $L^2$ norm.
\end{abstract}

\maketitle

\section{Introduction}
We consider the compressible isentropic bipolar Euler-Maxwell system in three space dimensions \cite{C,MRS,RG}
\begin{equation}  \label{yiyi}
\left\{
\begin{array}{lll}
\partial_t \tilde n_{\pm}+{\rm div} (\tilde n_{\pm} \tilde u_{\pm})=0,\\
\partial_t (\tilde n_{\pm} \tilde u_{\pm})+{\rm div}(\tilde n_{\pm} \tilde u_{\pm}\otimes \tilde u_{\pm}) + \na p_{\pm}(\tilde n_{\pm})=\pm\tilde n_\pm(\tilde{E}+\varepsilon \tilde u_\pm\times \tilde{B})- \frac{1}{\tau_\pm} \tilde n_\pm \tilde u_\pm,\\
\varepsilon\lambda^2\partial_t \tilde{E}-\na \times \tilde{B}= \varepsilon\left(\tilde n_- \tilde u_--\tilde n_+ \tilde u_+\right),\\
\varepsilon\partial_t \tilde{B}+\na \times \tilde{E}=0,\\
\lambda^2{\rm div}\tilde{E}= \tilde n_+  -\tilde n_-,\ \ {\rm div}\tilde{B}=0,\\
(\tilde n_\pm,\tilde u_\pm,\tilde{E},\tilde{B})|_{t=0}=(\tilde n_{\pm0},\tilde u_{\pm0}, \tilde{E}_0, \tilde{B}_0).
\end{array}
\right.
\end{equation}
Here the unknown functions are the charged density $\tilde n_\pm$, the velocity $\tilde u_\pm$, the electric field $\tilde{E}$ and the magnetic field $\tilde{B}$, with the subscripts $+$ and $-$ representing ion and electron respectively. We assume the pressure $p_\pm(\tilde n_\pm)=A_\pm\tilde n_{\pm}^{\gamma}$
with constants $A_\pm>0$ and $\gamma\ge 1$  the adiabatic exponent. $1/\tau_\pm>0$ are the velocity relaxation time of ions and electrons respectively. $\lambda>0$ is the Debye length, and $\varepsilon=1/c$ with $c$ the speed of light.

Although its significance in plasma physics and semiconductor physics, there are merely few mathematical results about the compressible Euler-Maxwell system since its complexity in mathematics. For the unipolar case:  Chen, Jerome and Wang \cite{CJW} showed the global existence of entropy weak solutions to the initial-boundary value problem for arbitrarily large initial data in $L^{\infty}(\mathbb{R})$; Guo and Tahvildar-Zadeh \cite{GT}
showed a blow-up criterion for spherically symmetric Euler-Maxwell system;
Recently, there are some results on the global existence
and the asymptotic behavior of smooth solutions with small amplitudes, see Tan et al. \cite{TWW}, Duan \cite{D}, Ueda and Kawashima \cite{UK},  Ueda et al. \cite{UWK}; For the asymptotic limits that derive simplified models starting from the Euler-Maxwell system, we refer to \cite{HP,PWG,X} for the relaxation limit, \cite{X} for the non-relativistic limit, \cite{PW1,PW2} for the quasi-neutral limit, \cite{T1,T2} for WKB asymptotics and the references therein.
For the bipolar case: Duan et al. \cite{DLZ} showed the global existence and time-decay rates of solutions near constant steady states with the vanishing electromagnetic field; Xu et al. \cite{XXK} studied the well-posedness in critical Besov spaces. Since the unipolar or bipolar Euler-Maxwell system is a symmetrizable hyperbolic system, the Cauchy problem in $\r3$ has a local unique smooth solution when the initial data is smooth, see Kato \cite{K} and Jerome \cite{J} for instance. Besides, we can refer to \cite{FWK,WFL} for the non-isentropic case.

In this paper, we will derive a refined global existence of smooth solutions near the constant equilibrium $(n_\infty,n_\infty,0,0,0,B_\infty)$ to the compressible isentropic bipolar Euler-Maxwell system and show some various time decay rates of the solution as well as its spatial derivatives of any order. Because of the complexities and some new difficulties, we will study the compressible non-isentropic bipolar Euler-Maxwell system in the future work. We should notice that the relaxation term of the velocity plays an important role in the whole paper. The non-relaxation case is much more difficult, we refer to \cite{GM,G,GP} for such a case. For the compressible unipolar Euler-Maxwell system \cite{TWW}, we do not need that the initial electron density belongs to negative Sobolev spaces $\dot{H}^{-s}$ or negative Besov spaces $\dot{B}^{-s}_{2,\infty}$ when deriving the optimal decay rates of solutions. However, in Theorem \ref{decay} the initial total densities $n_{10}$ must belong to $\dot{H}^{-s}$ or $\dot{B}^{-s}_{2,\infty}$ since the cancelation between two carriers. In fact, in Theorem \ref{decay} the assumption for the initial difference of densities $n_{20}$ could be deleted given \cite{TWW}. Compared with \cite{TWW}, there are two major difficulties except the computational complexity. First of all, the bipolar system \eqref{yiyi} could be reformulated equivalently as the damped Euler equations coupled with the one-fluid Euler-Maxwell equations \eqref{yi}. Then, the total densities $n_1$ in the damped Euler equations is degenerately dissipative because of the cancelation between two carriers. It is difficult to close the energy estimates since the degenerate dissipation of $n_1$. We manage to obtain the effective energy estimates by dealing carefully with these terms involved with $n_1$ in the proofs of Lemma \ref{energy lemma} and Lemma \ref{other di}. The other difficulty is caused by the nonlinear function $f(\frac{n_1\pm n_2}{2})$. Since $n_1$ and $n_2$ have different dissipative structures, we must be careful about the function $f(\frac{n_1\pm n_2}{2})$. Here we overcome such a obstacle by some detailed calculi. Without loss of generality, we take all the physical constants $\tau_\pm,\varepsilon,\lambda ,A_\pm, n_\infty$ in \eqref{yiyi} to be one.

We define
\begin{equation}\label{scaling}
\left\{
\begin{array}{lll}
n_\pm(x,t)=\frac{2}{\gamma-1}\left\{\left[\tilde n_\pm\Big(x,\frac{t}{\sqrt{\gamma}}\Big)\right]^{\frac{\gamma-1}{2}}-1\right\},\quad u_\pm(x,t)=\frac{1}{\sqrt{\gamma}}\tilde{u}_\pm\Big(x,\frac{t}{\sqrt{\gamma}}\Big), \\
E(x,t)=\frac{1}{\sqrt{\gamma}}\tilde{E}\Big(x,\frac{t}{\sqrt{\gamma}}\Big), \qquad B(x,t)=\frac{1}{\sqrt{\gamma}}\tilde{B}\Big(x,\frac{t}{\sqrt{\gamma}}\Big)-B_{\infty}.
\end{array}
\right.
\end{equation}
Then the Euler-Maxwell system \eqref{yiyi} is reformulated equivalently as
\begin{equation}
\left\{
\begin{array}{lll}
\displaystyle\partial_tn_\pm+{\rm div} u_\pm=-u_\pm\cdot\na n_\pm
-\mu n_\pm{\rm div} u_\pm,   \\
\displaystyle\partial_tu_\pm+\nu  u_\pm\mp u_\pm\times B_{\infty}+\na n_\pm\mp\nu  E=-u_\pm\cdot\na u_\pm
-\mu n_\pm\na n_\pm\pm u_\pm\times B, \\
\partial_t E-\nu\na \times B+\nu\left(u_+-u_-\right)= \nu\left(f(n_-) u_-- f(n_+)u_+\right),\\
\partial_t B+\nu  \na \times E=0,  \\
{\rm div} E=\nu  \left(f(n_+)-f(n_-)\right),\ \ {\rm div} B=0, \\
(n_\pm,u_\pm,E,B)|_{t=0}=(n_{\pm0},u_{\pm0},E_0,B_0).
\end{array}
\right.\nonumber
\end{equation}
Here $\mu:=\frac{\gamma-1}{2}$, $\nu:=\frac{1}{\sqrt{\gamma}}$ and the nonlinear function $f(n_\pm)$ is defined by
\begin{equation}  \label{f}
f(n_\pm):=\left(1+ \frac{\gamma-1}{2} n_\pm\right)^{\frac{2}{\gamma-1}}-1.
\end{equation}
In fact, we have assumed $\gamma>1$ in \eqref{scaling}. If $\gamma=1$, we instead define
$n_\pm:=\ln\tilde{n}_\pm.$

Let
\begin{equation}
\begin{split}
n_1=n_++n_-,\ n_2=n_+-n_-,\ u_1=u_++u_-,\ u_2=u_+-u_-,\nonumber
\end{split}
\end{equation}
that is
\begin{equation}\label{n1 n2}
\begin{split}
n_+=\frac{n_1+n_2}{2},\ n_-=\frac{n_1-n_2}{2},\ u_+=\frac{u_1+u_2}{2},\ u_-=\frac{u_1-u_2}{2}.
\end{split}
\end{equation}
Then $U:=(n_1,n_2,u_1,u_2,E,B)$ satisfies
\begin{equation}  \label{yi}
\left\{
\begin{array}{lll}
\displaystyle\partial_tn_1+{\rm div} u_1=g_1,   \\
\displaystyle\partial_tu_1+\nu  u_1-u_2\times B_{\infty}+\na n_1=g_2+u_2\times B, \\
\displaystyle\partial_tn_2+{\rm div} u_2=g_3,   \\
\displaystyle\partial_tu_2+\nu  u_2-u_1\times B_{\infty}+\na n_2-2\nu  E=g_4+u_1\times B, \\
\partial_t E-\nu  \na \times B+\nu u_2= g_5, & & \\
\partial_t B+\nu  \na \times E=0,  \\
{\rm div} E=\nu  \left(f(\frac{n_1+n_2}{2})-f(\frac{n_1-n_2}{2})\right),\ \ {\rm div} B=0,
\end{array}
\right.
\end{equation}
with initial data
$$U|_{t=0}=U_0:=(n_{10},n_{20},u_{10},u_{20},E_0,B_0).$$
Here
\begin{equation}\label{g}
\begin{split}
g_1&=-\frac{1}{2} \left(u_1\cdot\na n_1+u_2\cdot\na n_2\right)-\frac{\mu}{2}\left(n_1{\rm div}u_1+n_2{\rm div} u_2\right),\\
g_2&=-\frac{1}{2} \left(u_1\cdot\na u_1+u_2\cdot\na u_2\right)-\frac{\mu}{2}\left(n_1\na n_1+n_2\na n_2\right),\\
g_3&=-\frac{1}{2} \left(u_1\cdot\na n_2+u_2\cdot\na n_1\right)-\frac{\mu}{2}\left(n_1{\rm div}u_2+n_2{\rm div} u_1\right),\\
g_4&=-\frac{1}{2} \left(u_1\cdot\na u_2+u_2\cdot\na u_1\right)-\frac{\mu}{2}\left(n_1\na n_2+n_2\na n_1\right),\\
g_5&=\nu \left( f\left(\frac{n_1-n_2}{2}\right) \frac{u_1-u_2}{2}- f\left(\frac{n_1+n_2}{2}\right)\frac{u_1+u_2}{2}\right).
\end{split}
\end{equation}

\smallskip\smallskip

\noindent \textbf{Notations:} In this paper, we use $H^{s}(\mathbb{R}^{3}),s\in \mathbb{R}$ to denote the usual
Sobolev spaces with norm $\norm{\cdot}_{H^{s}}$ and $L^{p}(\mathbb{R}^{3}),1\leq p\leq
\infty $ to denote the usual $L^{p}$ spaces with norm $
\norm{\cdot}_{L^{p}}$.
$\na ^{\ell }$ with an
integer $\ell \geq 0$ stands for the usual any spatial derivatives of order $
\ell $. When $\ell <0$ or $\ell $ is not a positive integer, $\na ^{\ell }
$ stands for $\Lambda ^{\ell }$ defined by $
\Lambda^\ell f := \mathscr{F}^{-1} (|\xi|^\ell  \mathscr{F}{f}) $, where $\mathscr{F}$ is the usual Fourier transform operator and $\mathscr{F}^{-1}$ is its inverse. We use $\dot{H}
^{s}(\mathbb{R}^{3}),s\in \mathbb{R}$ to denote the homogeneous Sobolev
spaces on $\mathbb{R}^{3}$ with norm $\norm{\cdot}_{\dot{H}^{s}}$ defined by
$\norm{f}_{\dot{H}^s}:=\norm{\Lambda^s f}_{L^2}$. We then recall the homogeneous Besov spaces. Let
$\phi \in C^\infty_c(R^3_\xi)$ be such that $\phi(\xi) = 1$ when $|\xi| \le 1$ and $\phi(\xi) = 0$ when $|\xi| \ge 2$.  Let
$\varphi(\xi) = \phi(\xi) - \phi(2\xi)$ and $\varphi_j(\xi) = \varphi( 2^{-j}\xi)$
for ${j \in \mathbb {Z}}$. Then by the
construction,
$\sum_{j\in\mathbb {Z}}\varphi_j(\xi)=1$ if $\xi\neq 0.$
We define $\dot{\Delta}_j
f:= \mathscr{F}^{-1}(\varphi_j)* f$, then for $s\in \mathbb{R}$ and $1\le p,r\le \infty$, we define the homogeneous Besov spaces $\dot{B}_{p,r}^{s}(\r3)$ with norm $\norm{\cdot}_{\dot{B}_{p,r}^{s}}$ defined by
$$\|f\|_{\dot{B}_{p,r}^{s}}:=\Big(\sum\limits_{j\in\mathbb{Z}}2^{rsj}\|\dot{\Delta}_j f\|_{L^{p}}^{r}\Big)^{\frac1r}.$$Particularly, if $r=\infty$, then
\begin{equation}
\|f\|_{\dot{B}_{p,\infty}^{s}}:=\sup\limits_{j\in\mathbb{Z}}2^{sj}\norm{\dot{\Delta}_j f}_{L^{p}}.\nonumber
\end{equation}

Throughout this paper  we let $C$  denote
some positive (generally large) universal constants and $\lambda$ denote  some positive (generally small) universal constants. They {\it do not} depend on either $k$ or $N$; otherwise, we will denote them by $C_k$, $C_{N}$, etc.
We will use $a \lesssim b$ if $a \le C b$, and $a\thicksim b$ means that $a\lesssim b$ and $b\lesssim a$.
We use $C_0$ to denote the constants depending on the initial data and $k,N,s$.
For simplicity, we write $\norm{(A,B)}_{X}:=\norm{A}_{X} +\norm{B}_{X}$ and $\int f:=\int_{\mathbb{R}^3}f\,dx.$ $(\ast)\times\varepsilon+(\ast\ast)$ denote that multiplying $(\ast)$ by a sufficiently small but fixed factor $\varepsilon$ and then adding it to $(\ast\ast)$.

\smallskip\smallskip

For $N\ge 3$, we define the energy functional by
\begin{equation}
\mathcal{E}_N(t):=\sum_{l=0}^{N}\norm{ \na^{l}U}_{L^2}^2\nonumber
\end{equation}
and the corresponding dissipation rate by
\begin{equation}
\mathcal{D}_N(t):=\sum_{l=1}^{N}\norm{\na^{l}n_1}_{L^2}^2+\sum_{l=0}^{N}\norm{\na^{l}(n_2,u_1,u_2)}_{L^2}^2+\sum_{l=0}^{N-1}\norm{\na^{l}E}_{L^2}^2+ \sum_{l=1}^{N-1}\norm{\na^{l} B}_{L^2}^2.\nonumber
\end{equation}

Our first main result about the global unique solution to the system \eqref{yi} is stated as follows.
\begin{theorem}\label{existence}
Assume the initial data satisfy the compatible conditions
\begin{equation}
{\rm div}  {E}_0=\nu  \left(f(\frac{n_{10}+n_{20}}{2})-f(\frac{n_{10}-n_{20}}{2})\right),\ \ {\rm div}  {B}_0=0.\nonumber
\end{equation}
There exists a sufficiently small $\delta_0>0$ such that if $\mathcal{E}_3(0)\le \delta_0$, then there exists a unique global solution
$U(t)$ to the Euler-Maxwell system \eqref{yi} satisfying
\begin{equation}\label{energy inequality}
\sup_{0\leq t\leq \infty }\mathcal{ E}_3(t)+\int_{0}^{\infty }%
\mathcal{D}_3(\tau)\,d\tau\leq C\mathcal{ E}_3(0).
\end{equation}

Furthermore, if $\mathcal{E}_N(0)<+\infty$ for any $N\ge 3$, there exists an increasing continuous function $P_N(\cdot)$ with $P_N(0)=0$ such that the unique solution satisfies
\begin{equation}\label{energy inequality N}
\sup_{0\leq t\leq \infty }\mathcal{ E}_N(t)+\int_{0}^{\infty }%
\mathcal{D}_N(\tau)\,d\tau\leq P_N\left(\mathcal{ E}_N(0)\right).
\end{equation}
\end{theorem}

 In the proof of Theorem \ref{existence}, the major difficulties are caused by the degenerate dissipation for the total densities and the regularity-loss of the electromagnetic field. We will do the refined energy estimates stated in Lemma \ref{energy lemma}--\ref{other di}, which allow us to deduce
\begin{equation}
\frac{d}{dt} \mathcal{E} _3+\mathcal{D}_3\lesssim\sqrt{\mathcal{E} _3}\mathcal{D}_3\nonumber
\end{equation}
and for $N\ge 4$,
\begin{equation}
\frac{d}{dt} {\mathcal{E}}_N+ \mathcal{D}_N
 \le  C_{N } {\mathcal{D}_{N-1}} {\mathcal{E}_N}.\nonumber
\end{equation}
Then Theorem \ref{existence} follows in the fashion of \cite{G12,W12,TWW}.

Our second main result is on some various decay rates of the solution to the system \eqref{yi} by making the much stronger assumption on the initial data.
\begin{theorem}\label{decay}
Assume that $U(t)$ is the solution to the
Euler-Maxwell system \eqref{yi} constructed in Theorem \ref{existence} with $
N\geq 5$. There exists a sufficiently small $\delta_0=\delta_0(N)$ such that if $\mathcal{ {E}}_N(0)\le \delta_0$,
and assuming that $U_0\in \dot{H}^{-s}$ for some $s\in [0,3/2)$ or $U_0\in \dot{B}_{2,\infty}^{-s}$ for some $s\in (0,3/2]$,  then we have
\begin{equation}\label{H-sbound}
\norm{U(t)}_{\dot{H}^{-s}}\le C_0
\end{equation}or
\begin{equation}\label{H-sbound Besov}
\norm{U(t)}_{\dot{B}_{2,\infty}^{-s}}\le C_0.
\end{equation}
Moreover, for any fixed integer $k\ge 0$, if $N\ge 2k+2+s$, then
\begin{equation}\label{basic decay}
\norm{\na^kU(t)}_{L^2}\le  C_0 (1+ t)^{- \frac{k+s}{2} }.
\end{equation}

Furthermore, for any fixed integer $k\ge 0$, if $N\ge2k+4+s$, then
\begin{equation}\label{further decay1}
\norm{\na^k(n_2,u_1,u_2,E)(t)}_{L^2}\le  C_0 (1+ t)^{- \frac{k+1+s}{2} };
\end{equation}
if $N\ge2k+6+s$, then
\begin{equation}\label{further decay11}
\norm{\na^k n_2 (t)}_{L^2}\le  C_0 (1+ t)^{- \frac{k+2+s}{2} };
\end{equation}
if $N\ge2k+12+s$ and $ B_\infty =0$, then
\begin{equation}\label{further decay2}
\norm{\na^k (n_2,{\rm div}u_2) (t)}_{L^2}\le  C_0 (1+ t)^{- (\frac k2+\frac74+s) }.
\end{equation}
\end{theorem}

In the proof of Theorem \ref{decay}, we mainly use the regularity interpolation method developed in  Strain and Guo \cite{SG06}, Guo and Wang \cite{GW} and Sohinger and Strain \cite{SS}. To prove the optimal decay rate of the
dissipative equations in the whole space, Guo and Wang \cite{GW} developed a general energy
method of using a family of scaled energy estimates with minimum derivative counts and interpolations
among them. However, this method can not be applied directly to the compressible bipolar Euler-Maxwell system which is of regularity-loss. To overcome this obstacle caused by the regularity-loss of the electromagnetic field, we deduce from Lemma \ref{energy lemma}--\ref{other di} that
\begin{equation}
\frac{d}{dt} {\mathcal{E}}_k^{k+2}+\mathcal{D}_k^{k+2}\le C_k   \norm{(n_2,u_1,u_2)}_{L^\infty}\norm{\na^{k+2}(n_1,n_2,u_1,u_2)}_{L^2}
\norm{ \na^{k+2}( E,  B )}_{L^2},\nonumber
\end{equation}
where ${\mathcal{E}}_k^{k+2}$ and $\mathcal{D}_k^{k+2}$ with minimum derivative counts are defined by \eqref{1111} and \eqref{2222} respectively.
Then combining the methods of \cite{GW,SS} and a trick of Strain and Guo \cite{SG06} to treat the electromagnetic field, we manage to conclude the decay rate \eqref{basic decay}. If in view of the whole solution, the decay rate \eqref{basic decay} can be regarded as be optimal. The higher decay rates \eqref{further decay1}--\eqref{further decay2} follow by revisiting the equations carefully. In particular, we will use a bootstrap argument to derive \eqref{further decay2}.

By Theorem \ref{decay} and Lemma \ref{Riesz lemma}--\ref{Lp embedding}, we have the following corollary of the usual $L^p$--$L^2$ type of the decay results:
\begin{Corollary}\label{2mainth}
Under the assumptions of Theorem \ref{decay} except that we replace the $\dot{H}^{-s}$ or $\Dot{B}_{2,\infty}^{-s}$ assumption  by that $U_0\in L^p$ for some $p\in [1,2]$, then for any fixed integer $k\ge 0$, if $N\ge 2k+2+ s_p $, then
\begin{equation}
\norm{\na^kU(t)}_{L^2}\le  C_0 (1+ t)^{- \frac{k+s_p}{2} }.\nonumber
\end{equation}
Here the number $s_{p} :=3\left(\frac{1}{p}-\frac{1}{2}\right)$.

Furthermore, for any fixed integer $k\ge 0$, if $N\ge2k+4+s_p$, then
\begin{equation}
\norm{\na^k(n_2,u_1,u_2,E)(t)}_{L^2}\le  C_0 (1+ t)^{- \frac{k+1+s_p}{2} };\nonumber
\end{equation}
if $N\ge2k+6+s_p$, then
\begin{equation}
\norm{\na^k n_2 (t)}_{L^2}\le  C_0 (1+ t)^{- \frac{k+2+s_p}{2} };\nonumber
\end{equation}
if $N\ge2k+12+s_p$ and $ B_\infty =0$, then
\begin{equation}\label{n L^1}
\norm{\na^k (n_2,{\rm div}u_2) (t)}_{L^2}\le  C_0 (1+ t)^{- (\frac k2+\frac74+s_p) }.
\end{equation}
\end{Corollary}

The followings are several remarks on Theorem \ref{existence}--\ref{decay} and Corollary \ref{2mainth}.

\begin{Remark}
In Theorem \ref{existence}, we only assume that the initial data is small in the $H^3$ norm but the higher order derivatives could be large.
Notice that in Theorem \ref{decay} the $\dot{H}^{-s}$ and $\dot{B}_{2,\infty}^{-s}$ norms of the solution are preserved along the time evolution, however, in Corollary \ref{2mainth} it is difficult to show that the $L^p$ norm of the solution can be preserved.
Note that the $L^2$ decay rate of the higher order spatial derivatives of the solution
is obtained. Then the general optimal $L^q$ $(2\le q\le \infty)$ decay rates of the solution follow by the Sobolev
interpolation.
\end{Remark}

\begin{Remark}
In Theorem \ref{decay}, the space $\Dot{H}^{-s}$ or $\Dot{B}_{2,\infty}^{-s}$ was introduced there to enhance the decay rates.
By the usual embedding theorem, we know that for $p\in (1,2]$, $L^p\subset \Dot{H}^{-s}$ with $s=3(\frac{1}{p}-\frac{1}{2})\in[0,3/2)$. Meantime, we note that the endpoint embedding  $L^1\subset \Dot{B}_{2,\infty}^{-\frac{3}{2}}$ holds. Hence the $L^p$--$L^2(1\le p\le2)$ type of the optimal decay results follows as a corollary.
\end{Remark}

\begin{Remark}
We remark that Corollary \ref{2mainth} not only provides an alternative approach to derive the $L^p$--$L^2$ type
of the optimal decay results but also improves the previous results of the  $L^p$--$L^2$ approach in Duan et al. \cite{DLZ}.
In Duan et al. \cite{DLZ}, assuming that $B_\infty=0$ and $\norm{U_0}_{L^1}$ is sufficiently small, by combining the energy method and the linear decay analysis, Duan proved that
\begin{equation}
\norm{n_2 (t)}_{L^2}\le  C_0 (1+ t)^{- \frac{5}{2} },\quad \norm{(u_1,u_2,E)(t)}_{L^2}\le  C_0 (1+ t)^{- \frac{5}{4} } \text{ and } \norm{(n_1,B)(t)}_{L^2}\le  C_0 (1+ t)^{- \frac{3}{4} }.\nonumber
\end{equation}
Notice that for $p=1$, our decay rate of $n_2(t)$ is $(1+t)^{-13/4}$ in \eqref{n L^1}.
\end{Remark}

The rest of our paper is structured as follows. In section \ref{section2}, we establish the refined energy estimates for the solution and derive the negative Sobolev and Besov estimates. Theorem \ref{existence} and Theorem \ref{decay} are proved in section \ref{section3}.

\section{Nonlinear energy estimates}\label{section2}

In this section, we will do the a priori estimate by assuming that $\norm{n_\pm(t)}_{H^3}\le \delta\ll 1$. Recall the expression \eqref{f} of $f(n_\pm)$ and $\eqref{n1 n2}$. Then by Taylor's formula and Sobolev's inequality, we have
\begin{equation}\label{fn}
f\left(\frac{n_1\pm n_2}{2}\right)\sim \frac{n_1\pm n_2}{2}\hbox{ and } \norms{f^{(k)}\left(\frac{n_1\pm n_2}{2}\right)} \le C_k\hbox{ for any }k\ge 1.
\end{equation}

\subsection{Preliminary}

In this subsection, we collect some analytic tools used later in this paper.

 \begin{lemma}\label{A1}
Let $2\le p\le +\infty$ and $\alpha,m,\ell\ge 0$. Then we have
\begin{equation}
\norm{\na^\alpha f}_{L^p}\le C_{p} \norm{ \na^mf}_{L^2}^{1-\theta}
\norm{ \na^\ell f}_{L^2}^{\theta}.\nonumber
\end{equation}
Here $0\le \theta\le 1$ (if $p=+\infty$, then we require that $0<\theta<1$) and $\alpha$ satisfies
\begin{equation}
\alpha+3\left(\frac12-\frac{1}{p}\right)=m(1-\theta)+\ell\theta.\nonumber
\end{equation}
\end{lemma}
\begin{proof}
For the case $2\le p<+\infty$, we refer to Lemma A.1 in \cite{GW}; for the case $p=+\infty$, we refer to Exercise 6.1.2 in \cite{Gla} (pp. 421).
\end{proof}
\begin{lemma}\label{A2}
For any integer $k\ge0$, we have
\begin{equation}\label{fkn0}
\norm{\na^kf(n)}_{L^\infty} \le C_k\norm{\na^{k+1} n}
_{L^2}^{1/2}\norm{\na^{k+2}n}_{L^2} ^{1/2},
\end{equation}
and
\begin{equation}
\norm{\na^kf(n)}_{L^2}\le C_k\norm{\na^kn}_{L^2}.\nonumber
\end{equation}
\end{lemma}
\begin{proof}
See Lemma 2.2 in \cite{TWW}.
\end{proof}
We recall the following commutator estimate:
\begin{lemma}\label{commutator}
Let $k\ge 1$ be an integer and define the commutator
\begin{equation}
\left[\na^k,g\right]h=\na^k(gh)-g\na^kh.\nonumber
\end{equation}
Then we have
\begin{equation}
\norm{\left[\na^k,g\right]h}_{L^2} \le C_k\left( \norm{\na g}_{L^\infty}
\norm{\na^{k-1}h}_{L^2}+\norm{\na^k g}_{L^2}\norm{ h}_{L^\infty}\right),\nonumber
\end{equation}and
\begin{equation}
\norm{\na^k(gh)}_{L^2} \le C_k\left( \norm{g}_{L^\infty}
\norm{\na^{k}h}_{L^2}+\norm{\na^k g}_{L^2}\norm{ h}_{L^\infty}\right).\nonumber
\end{equation}
\end{lemma}
\begin{proof}
It can be proved by using Lemma \ref{A1}, see Lemma 3.4 in \cite{MB} (pp. 98) for instance.
\end{proof}
Notice that when using the commutator estimate in this paper, we usually will not consider the case that $k=0$ since it is trivial.

We have the $L^p$ embeddings:
\begin{lemma}\label{Riesz lemma}
Let $0\le s<3/2,\ 1<p\le 2$ with $1/2+s/3=1/p$, then
\begin{equation}
\norm{ f}_{\dot{H}^{-s}}\lesssim\norm{ f}_{L^p}.\nonumber
\end{equation}
\end{lemma}
\begin{proof}
It follows from the Hardy-Littlewood-Sobolev theorem, see \cite{Gla}.
\end{proof}

\begin{lemma}\label{Lp embedding}
Let $0< s\le 3/2,\ 1\le p<2$ with $1/2+s/3=1/p$, then
\begin{equation}
\norm{f}_{\dot{B}_{2,\infty}^{-s}}\lesssim\norm{f}_{L^p}.\nonumber
\end{equation}
\end{lemma}
\begin{proof}
See Lemma 4.6 in \cite{SS}.
\end{proof}

It is important to use the following special interpolation estimates:
\begin{lemma}\label{1-sinte}
Let $s\ge 0$ and $\ell\ge 0$, then we have
\begin{equation}
\norm{\na^\ell f}_{L^2}\le \norm{\na^{\ell+1} f}_{L^2}^{1-\theta}%
\norm{ f}_{\dot{H}^{-s}}^\theta, \hbox{ where }\theta=\frac{1}{\ell+1+s}.\nonumber
\end{equation}
\end{lemma}
\begin{proof}
It follows directly by the Parseval theorem  and H\"older's
inequality.
\end{proof}

\begin{lemma}\label{Besov interpolation}
Let $s> 0$ and $\ell\ge 0$, then we have
\begin{equation}
\norm{\na^\ell f}_{L^2}\le \norm{\na^{\ell+1} f}_{L^2}^{1-\theta}
\norm{ f}_{\dot{B}^{-s}_{2,\infty}}^\theta, \hbox{ where }\theta=\frac{1}{\ell+1+s}.\nonumber
\end{equation}
\end{lemma}
\begin{proof}
See Lemma 4.5 in \cite{SS}.
\end{proof}

\subsection{Energy estimates}

In this subsection, we will derive the basic energy estimates for the solution to the Euler-Maxwell system \eqref{yi}. We begin with the standard energy estimates.
\begin{lemma}\label{energy lemma}
For any integer $k\ge0$, we have
\begin{eqnarray}\label{energy 1}
&&\frac{d}{dt}\sum_{l=k}^{k+2}\norm{\na^{l}U}_{L^2}^2 +\la\sum_{l=k}^{k+2}\norm{\na^{l} (u_1,u_2)}_{L^2}^2 \nonumber\\
&&\quad\lesssim C_k F
\left(\sum_{l=k+1}^{k+2}\norm{\na^{l} n_1}_{L^2}^2+\sum_{l=k}^{k+2}\norm{\na^{l} (n_2, u_1,u_2 )}_{L^2}^2+\sum_{l=k}^{k+1}\norm{\na^{l}E}_{L^2}^2+\norm{\na^{k+1} B}_{L^2}^2\right)\nonumber\\
&&\qquad+\norm{(n_2,u_1,u_2)}_{L^\infty}\norm{\na^{k+2} (n_2,u_1,u_2)}_{L^2} \norm{\na^{k+2} ( E,  B )}_{L^2},
\end{eqnarray}
where $F$ is defined by
$$F=F(n_1,n_2,u_1,u_2,B):=\norm{\na n_1}_{H^2}+\norm{ (n_2,u_1,u_2)}_{ H^{k+1}\cap H^{\frac{k}{2}+2}\cap H^3}+\norm{\na B}_{L^2}.$$
\end{lemma}
\begin{proof}
The standard $\na^l$ ($l=k,k+1,k+2$) energy estimates on the system $\eqref{yi}$ yield
\begin{eqnarray}\label{I1-I3}
&&\frac{1}{2}\frac{d}{dt}\int \norms{\na^{l}(n_1,n_2,u_1,u_2)}^2+\frac{d}{dt}\int \norms{\na^{l}(E,B)}^2   +\nu  \norm{\na^l (u_1,u_2)}_{L^2}^2 \nonumber\\
&&\quad=\int \na^l g_1\na^ln_1+\na^lg_2\cdot\na^l u_1
+ \na^l g_3\na^ln_2+ \na^lg_4\cdot\na^l u_2 \nonumber\\
&&\qquad+\int \na^l (u_2\times B)\cdot\na^l u_1+ \na^l (u_1\times B)\cdot\na^l u_2+2\nu \int \na^{l} g_5\cdot\na^{l} E \nonumber\\
&&\quad:=I_1+I_2+2\nu I_3.
\end{eqnarray}

We now estimate $I_1\sim I_3$. First, by \eqref{g}, we split $I_1$ as:
\begin{eqnarray}\label{I1}
I_1&&=-\frac12\int \na^l\left(u_1\cdot\na n_1\right)\na^ln_1+\na^l\left(u_1\cdot\na u_1\right)\cdot\na^lu_1\nonumber\\
&&\quad-\frac12\int \na^l\left(u_1\cdot\na n_2\right)\na^ln_2+\na^l\left(u_1\cdot\na u_2\right)\cdot\na^lu_2\nonumber\\
&&\quad-\frac12\int \na^l\left(u_2\cdot\na n_2\right)\na^ln_1+\na^l\left(u_2\cdot\na n_1\right)\na^ln_2\nonumber\\
&&\quad-\frac12\int \na^l\left(u_2\cdot\na u_2\right)\cdot\na^lu_1+\na^l\left(u_2\cdot\na u_1\right)\cdot\na^lu_2\nonumber\\
&&\quad-\frac{\mu}{2}\int \na^l\left(n_1{\rm div}u_1\right)\na^ln_1+\na^l\left(n_1\na n_1\right)\cdot\na^lu_1\nonumber\\
&&\quad-\frac{\mu}{2}\int \na^l\left(n_1{\rm div}u_2\right)\na^ln_2+\na^l\left(n_1\na n_2\right)\cdot\na^lu_2\nonumber\\
&&\quad-\frac{\mu}{2}\int \na^l\left(n_2{\rm div}u_2\right)\na^ln_1+\na^l\left(n_2\na n_1\right)\cdot\na^lu_2\nonumber\\
&&\quad-\frac{\mu}{2}\int \na^l\left(n_2{\rm div}u_1\right)\na^ln_2+\na^l\left(n_2\na n_2\right)\cdot\na^lu_1\nonumber\\
&&:=\frac12\left(I_{11}+I_{12}+I_{13}+I_{14}\right)+\frac{\mu}{2}\left(I_{15}+I_{16}+I_{17}+I_{18}\right).
\end{eqnarray}

We shall estimate the eight terms on the right-hand side of \eqref{I1}. We must be careful about these terms involving $n_1$ since $n_1$ is degenerately dissipative. First we estimate $I_{11}$. We have to distinct the arguments by the value of $l$. For $l=k$ or $k+1$, we have
\begin{eqnarray}  \label{Iyyyi}
-\int \na^l\left(u_1\cdot\na n_1\right)\na^ln_1&&=-\int\sum_{0\le \ell\le
l}C_{l}^\ell\na^{l-\ell} u_{1}\cdot\na\na^{\ell}n_{1}
\na^{l}n_{1} \nonumber\\
&&\lesssim\sum_{0\le \ell\le
l}\norm{\na^{l-\ell} u_{1}\cdot\na\na^{\ell}n_{1}}_{L^{6/5}}\norm{\na^ln_{1}}_{L^6}\nonumber\\
&&\lesssim\sum_{0\le \ell\le
l}\norm{\na^{l-\ell} u_{1}\cdot\na\na^{\ell}n_{1}}_{L^{6/5}}\norm{\na^{l+1}n_{1}}_{L^2}.
\end{eqnarray}
If $0\le \ell\le \left[\frac l2\right]$, by H\"older's inequality and Lemma A.1, we have
\begin{eqnarray}\label{Iyyer}
\norm{\na^{l-\ell} u_{1}\cdot\na\na^{\ell}n_{1}}_{L^{6/5}}&&\lesssim \norm{\na^{l-\ell} u_{1}}_{L^2} \norm{\na^{\ell+1}n_{1}}_{L^3}\nonumber
\\&&\lesssim \norm{u_{1}}^{\frac \ell l}_{L^2}\norm{\na^lu_{1}}_{L^2}^{1-\frac{\ell}{l}} \norm{\na^\alpha n_{1}}_{L^2}^{1-\frac{\ell}{l}}\norm{\na^{l+1}n_{1}}_{L^2}^{\frac \ell l}\nonumber\\
&&\lesssim \left(\norm{\na n_1}_{H^2}+\norm{u_1}_{H^3}\right)\left(\norm{\na^{l+1}n_{1}}_{L^2}+\norm{\na^{l} u_{1}}_{L^2}\right),
\end{eqnarray}
where $\alpha$ is defined by
\begin{equation}
\begin{split}
\ell+\frac{3}{2}=\alpha\times\left(1-\frac{\ell}{l}\right)
+(l+1)\times \frac{\ell}{l} \Longrightarrow \alpha=\frac{3l-2\ell}{2l-2\ell}\in \left[\frac{3}{2},3\right) \, \text{ since }\ell\le \frac{l}{2};\nonumber
\end{split}
\end{equation}
if $\left[\frac l2\right]+1\le \ell\le l$, by H\"older's inequality and Lemma A.1 again, we have
\begin{eqnarray}\label{Iyysi}
\norm{\na^{l-\ell} u_{1}\cdot\na\na^{\ell}n_{1}}_{L^{6/5}}&&\lesssim \norm{\na^{l-\ell} u_{1}}_{L^3} \norm{\na^{\ell+1}n_{1}}_{L^2}\nonumber
\\&&\lesssim \norm{\na^\alpha u_{1}}^{\frac {\ell}{l}}_{L^2}\norm{\na^lu_{1}}_{L^2}^{\frac{l-\ell}{l}} \norm{\na n_{1}}_{L^2}^{\frac{l-\ell}{l}}\norm{\na^{l+1}n_{1}}_{L^2}^{\frac {\ell}{l}}\nonumber\\
&&\lesssim \left(\norm{\na n_1}_{H^2}+\norm{u_1}_{H^3}\right)\left(\norm{\na^{l+1}n_{1}}_{L^2}+\norm{\na^{l} u_{1}}_{L^2}\right),
\end{eqnarray}
where $\alpha$ is defined by
\begin{equation}
\begin{split}
l-\ell+\frac{1}{2}=\alpha\times\frac{\ell}{l}
+l-\ell \Longrightarrow \alpha=\frac{l}{2\ell}\in \left[\frac{1}{2},3\right) \, \text{ since }\ell\ge \frac{l+1}{2}.\nonumber
\end{split}
\end{equation}
In light of \eqref{Iyyer} and \eqref{Iyysi}, we deduce from \eqref{Iyyyi} that for $l=k$ or $k+1$,
\begin{equation}\label{Iyyliu}
-\int \na^l\left(u_1\cdot\na n_1\right)\na^ln_1\lesssim\left(\norm{\na n_1}_{H^2}+\norm{u_1}_{H^3}\right)\left(\norm{\na^{l+1}n_{1}}_{L^2}^2 +\norm{\na^{l}u_{1}}_{L^2}^2 \right).
\end{equation}
Now for $l=k+2$, by integrating by parts and Lemma \ref{commutator}, we have
\begin{eqnarray}  \label{Iyyqi}
-\int\na^{k+2} (u_1\cdot\na n_1)\na^{k+2}n_1\nonumber
&&=-\int\left[\na^{k+2},u_1\right]\cdot\na n_1\na^{k+2}n_1-\int u_1\cdot\na\na^{k+2} n_1\na^{k+2}n_1 \\\nonumber
&&\lesssim\left(\norm{\na u_1}_{L^\infty}\norm{\na^{k+2}n_1}_{L^2}+\norm{\na^{k+2}u_1}_{L^2}\norm{\na n_1}_{L^\infty}\right)\norm{\na^{k+2}n_1}_{L^2}\\\nonumber
&&\quad-\frac12\int u_1\cdot\na\Big(\na^{k+2} n_1\na^{k+2}n_1 \Big) \\\nonumber
&&\lesssim\norm{\na (n_1,u_1)}_{L^\infty}\norm{\na^{k+2}(n_1,u_1)}_{L^2}^2+\frac12\int{\rm div} u_1\left|\na^{k+2} n_1\right|^2\\
&&\lesssim\left(\norm{\na n_1}_{H^2}+\norm{u_1}_{H^3}\right) \norm{\na^{k+2}(n_1,u_1)}_{L^2}^2.
\end{eqnarray}
On the other hand, like \eqref{Iyyqi}, we have for $l=k$, $k+1$, $k+2$,
\begin{eqnarray}\label{Iyyba}
-\int \na^l\left(u_1\cdot\na u_1\right)\cdot\na^lu_1&&=-\int \left(u_1\cdot\na\na^lu_1+\left[\na^l,u_1\right]\cdot\na u_1\right)\cdot\na^lu_1\nonumber\\
&&=-\int \frac12u_1\cdot\na\left(\na^lu_1\cdot\na^lu_1\right)+\left[\na^l,u_1\right]\cdot\na u_1\cdot\na^lu_1\nonumber\\
&&\lesssim\norm{\na u_1}_{L^\infty}\norm{\na^lu_1}_{L^2}^2.
\end{eqnarray}
Hence, by \eqref{Iyyliu}--\eqref{Iyyba}, we have for $l=k$, $k+1$,
\begin{equation}
\begin{split}
I_{11}\lesssim\left(\norm{\na n_1}_{H^2}+\norm{u_1}_{H^3}\right)\left(\norm{\na^{l+1}n_1}_{L^2}^2+\norm{\na^lu_1}_{L^2}^2\right),\nonumber
\end{split}
\end{equation}
and for $l=k+2$,
\begin{equation}
\begin{split}
I_{11}\lesssim\left(\norm{\na n_1}_{H^2}+\norm{u_1}_{H^3}\right)\norm{\na^{k+2}(n_1,u_1)}_{L^2}^2.\nonumber
\end{split}
\end{equation}
Like \eqref{Iyyqi}, we have for $l=k$, $k+1$, $k+2$,
\begin{equation}
\begin{split}
I_{12}\lesssim\norm{(n_2,u_1,u_2)}_{H^3}\norm{\na^l(n_2,u_1,u_2)}_{L^2}^2,\
I_{14}\lesssim\norm{(u_1,u_2)}_{H^3}\norm{\na^{l}(u_1,u_2)}_{L^2}^2.\nonumber
\end{split}
\end{equation}
As in \eqref{Iyyyi}--\eqref{Iyyqi}, we have for $l=k$, $k+1$,
\begin{equation}
\begin{split}
I_{13}\lesssim\left(\norm{\na n_1}_{H^2}+\norm{(n_2,u_2)}_{H^3}\right)\left(\norm{\na^{l+1}(n_1,n_2)}_{L^2}^2+\norm{\na^lu_1}_{L^2}^2\right),\nonumber
\end{split}
\end{equation}
and for $l=k+2$,
\begin{equation}
\begin{split}
I_{13}\lesssim\left(\norm{\na n_1}_{H^2}+\norm{(n_2,u_2)}_{H^3}\right)\norm{\na^{k+2}(n_1,n_2,u_1)}_{L^2}^2.\nonumber
\end{split}
\end{equation}

We next estimate the term $I_{15}$. For $l=k$ or $k+1$, we split $I_{15}$ as:
\begin{eqnarray}  \label{Iywyi}
I_{15}&&=-\int\na^l (n_1\mathrm{div}u_1)\na^ln_1+\na^l(n_1\na n_1)\cdot\na^l u_1\nonumber\\
&&=-\int\sum_{0\le \ell\le
l}C_{l}^\ell\Big(\na^{l-\ell}n_1
\na^{\ell}{\rm div}u_1\na^{l}n_1+\na^{l-\ell}n_1
\na^{\ell+1}n_1\cdot\na^{l}u_1\Big) \nonumber\\
&&=-\int\sum_{0\le \ell\le
l-1}C_{l}^\ell\na^{l-\ell} n_1\na^{\ell}{\rm div}u_1
\na^{l}n_1 -\int\sum_{0\le \ell\le
l-1}C_{l}^\ell\na^{l-\ell} n_1\na^{\ell+1}n_1\cdot
\na^{l}u_1\nonumber\\
&&\quad -\int n_1{\rm div}\na^{l} u_1\na^ln_1+ n_1 \na^{l+1} n_1\cdot\na^l u_1\nonumber\\
&&:=I_{151}+I_{152}+I_{153}.
\end{eqnarray}
First we estimate $I_{153}$. By H\"older's, Sobolev's and Cauchy's inequalities, we obtain
\begin{eqnarray}  \label{Iywer}
I_{153}&&=-\int n_1{\rm div}\na^{l}u_1\na^ln_1+ n_1 \na^{l+1} n_1\cdot\na^l u_1=-\int n_1{\rm div}\Big(\na^{l} u_1\na^ln_1\Big)=\int\na n_1\na^{l}u_1\na^ln_1\nonumber\\
&&\lesssim\norm{\na n_1}_{L^3}\norm{\na^{l}u_1}_{L^2}\norm{\na^{l}n_1}_{L^6}\lesssim \norm{\na n_1}_{H^2} \left(\norm{\na^{l+1}n_1}_{L^2}^2+\norm{\na^{l} u_1}_{L^2}^2\right).
\end{eqnarray}
Next we estimate the term $I_{151}$. By H\"older's and Sobolev's inequalities, we obtain
\begin{eqnarray}  \label{Iywsan}
I_{151}&&=-\int\sum_{0\le \ell\le
l-1}C_{l}^\ell\na^{l-\ell} n_1\na^{\ell}{\rm div} u_1
\na^{l}n_1 \lesssim\sum_{0\le \ell\le
l-1}\norm{\na^{l-\ell} n_1\na^{\ell}{\rm div} u_1}_{L^{6/5}}\norm{\na^ln_1}_{L^6}\nonumber\\
&&\lesssim\sum_{0\le \ell\le
l-1}\norm{\na^{l-\ell} n_1\na^{\ell}{\rm div} u_1}_{L^{6/5}}\norm{\na^{l+1}n_1}_{L^2}.
\end{eqnarray}
If $0\le \ell\le \left[\frac l2\right]$, by H\"older's inequality and Lemma A.1, we have
\begin{eqnarray}\label{Iywsi}
\norm{\na^{l-\ell} n_1\na^{\ell}{\rm div} u_1}_{L^{6/5}}&&\lesssim \norm{\na^{l-\ell} n_1}_{L^3} \norm{\na^{\ell+1}u_1}_{L^2}\nonumber
\\&&\lesssim \norm{\na n_1}^{\frac {2\ell+1}{2l}}_{L^2}\norm{\na^{l+1}n_1}_{L^2}^{\frac{2l-2\ell-1}{2l}} \norm{\na^\alpha u_1}_{L^2}^{\frac{2l-2\ell-1}{2l}}\norm{\na^{l}u_1}_{L^2}^{\frac {2\ell+1}{2l}}\nonumber\\
&&\lesssim \left(\norm{\na n_1}_{L^2}+\norm{u_1}_{H^3}\right)\left(\norm{\na^{l+1}n_1}_{L^2}+\norm{\na^{l} u_1}_{L^2}\right),
\end{eqnarray}
where $\alpha$ is defined by
\begin{equation}
\begin{split}
\ell+1=\alpha\times\frac{2l-2\ell-1}{2l}
+\frac {2\ell+1}{2} \Longrightarrow \alpha=\frac{l}{2l-2\ell-1}\in \left(\frac{1}{2},3\right) \, \text{ since }\ell\le \frac{l}{2};\nonumber
\end{split}
\end{equation}
if $\left[\frac l2\right]+1\le \ell\le l-1$, by H\"older's inequality and Lemma A.1 again, we have
\begin{eqnarray}\label{Iywliu}
\norm{\na^{l-\ell} n_1\na^{\ell}{\rm div} u_1}_{L^{6/5}}&&\lesssim \norm{\na^{l-\ell} n_1}_{L^3} \norm{\na^{\ell+1}u_1}_{L^2}\nonumber
\\&&\lesssim \norm{\na^\alpha n_1}^{\frac {\ell+1}{l}}_{L^2}\norm{\na^{l+1}n_1}_{L^2}^{\frac{l-\ell-1}{l}} \norm{u_1}_{L^2}^{\frac{l-\ell-1}{l}}\norm{\na^{l}u_1}_{L^2}^{\frac {\ell+1}{l}}\nonumber\\
&&\lesssim \left(\norm{\na n_1}_{H^2}+\norm{u_1}_{H^3}\right)\left(\norm{\na^{l+1}n_1}_{L^2}+\norm{\na^{l} u_1}_{L^2}\right),
\end{eqnarray}
where $\alpha$ is defined by
\begin{equation}
\begin{split}
&l-\ell+\frac{1}{2}=\alpha\times\frac{\ell+1}{l}
+(l+1)\times \frac{l-\ell-1}{l}
\\ &\quad \Longrightarrow \alpha=1+\frac{l}{2\ell+2}\in \left[\frac{3}{2},3\right) \, \text{ since }\ell\ge \frac{l+1}{2}.\nonumber
\end{split}
\end{equation}
In light of \eqref{Iywsi} and \eqref{Iywliu}, we deduce from \eqref{Iywsan} that
\begin{equation}\label{Iywba}
I_{151}\lesssim\left(\norm{\na n_1}_{H^2}+\norm{u_1}_{H^3}\right)\left(\norm{\na^{l+1}n_1}_{L^2}^2 +\norm{\na^{l}u_1}_{L^2}^2 \right).
\end{equation}
Finally, we estimate the term $I_{152}$. By H\"older's inequality, we obtain
\begin{equation}  \label{Iywjiu}
\begin{split}
I_{152}=-\int\sum_{0\le \ell\le
l-1}C_{l}^\ell\na^{l-\ell} n_1\na^{\ell+1}n_1\cdot
\na^{l}u_1  \lesssim\sum_{0\le \ell\le
l-1}\norm{\na^{l-\ell} n_1\na^{\ell+1}n_1}_{L^{2}}\norm{\na^{l}u_1}_{L^2}.
\end{split}
\end{equation}
If $0\le \ell\le \left[\frac l2\right]$, by H\"older's inequality and Lemma A.1, we have
\begin{eqnarray}\label{Iywshi}
\norm{\na^{l-\ell} n_1\na^{\ell+1}n_1}_{L^{2}}&&\lesssim \norm{\na^{l-\ell} n_1}_{L^6} \norm{\na^{\ell+1}n_1}_{L^3}\nonumber
\\&&\lesssim \norm{\na n_1}^{\frac {\ell}{l}}_{L^2}\norm{\na^{l+1}n_1}_{L^2}^{\frac{l-\ell}{l}} \norm{\na^\alpha n_1}_{L^2}^{\frac{l-\ell}{l}}\norm{\na^{l+1}n_1}_{L^2}^{\frac{\ell}{l}}\nonumber\\
&&\lesssim \norm{\na n_1}_{H^2}\norm{\na^{l+1} n_1}_{L^2},
\end{eqnarray}
where $\alpha$ is defined by
\begin{equation}
\begin{split}
\ell+\frac{3}{2}=\alpha\times\frac{l-\ell}{l}
+(l+1)\times\frac{\ell}{l} \Longrightarrow \alpha=\frac{3l-2\ell}{2l-2\ell}\in \left[\frac{3}{2},3\right) \, \text{ since }\ell\le \frac{l}{2};\nonumber
\end{split}
\end{equation}
if $\left[\frac l2\right]+1\le \ell\le l-1$, by H\"older's inequality and Lemma A.1 again, we have
\begin{eqnarray}\label{Iywshier}
\norm{\na^{l-\ell} n_1\na^{\ell+1}n_1}_{L^{2}}&&\lesssim \norm{\na^{l-\ell} n_1}_{L^3} \norm{\na^{\ell+1}n_1}_{L^6}\nonumber
\\&&\lesssim \norm{\na^\alpha n_1}^{\frac {\ell}{l-1}}_{L^2}\norm{\na^{l+1}n_1}_{L^2}^{1-\frac{\ell}{l-1}} \norm{\na^2n_1}_{L^2}^{1-\frac{\ell}{l-1}}\norm{\na^{l+1}n_1}_{L^2}^{\frac {\ell}{l-1}}\nonumber\\
&&\lesssim \norm{\na n_1}_{H^2}\norm{\na^{l+1} n_1}_{L^2},
\end{eqnarray}
where $\alpha$ is defined by
\begin{equation}
\begin{split}
&l-\ell+\frac{1}{2}=\alpha\times\frac {\ell}{l-1}
+(l+1)\times \left(1-\frac{\ell}{l-1}\right)
\\ &\quad \Longrightarrow \alpha=2+\frac{-l+1}{2\ell}\in \left[\frac{3}{2},3\right) \, \text{ since }\ell\ge \frac{l+1}{2}.\nonumber
\end{split}
\end{equation}
In light of \eqref{Iywshi} and \eqref{Iywshier}, we deduce from \eqref{Iywjiu} that
\begin{equation}\label{Iywshisi}
I_{152}\lesssim\norm{\na n_1}_{H^2}\left(\norm{\na^{l+1}n_1}_{L^2}^2 +\norm{\na^{l}u_1}_{L^2}^2 \right).
\end{equation}
Hence, by \eqref{Iywer}, \eqref{Iywba} and \eqref{Iywshisi}, we deduce from \eqref{Iywyi} that for $l=k,k+1$,
\begin{equation}
I_{15}\lesssim\left(\norm{\na n_1}_{H^2}+\norm{u_1}_{H^3}\right)\left(\norm{\na^{l+1}n_1}_{L^2}^2 +\norm{\na^{l}u_1}_{L^2}^2 \right).\nonumber
\end{equation}
For $l=k+2$, like \eqref{Iyyqi}, we have
\begin{equation}
\begin{split}
I_{15}&=-\int\na^{k+2}(n_1{\rm div} u_1)\na^{k+2}n_1+\na^{k+2}(n_1\na n_1)\cdot\na^{k+2} u_1 \\
&=-\int\left[\na^{k+2},n_1\right]{\rm div} u_1\na^{k+2}n_1+\left[\na^{k+2},n_1\right]\na n_1\cdot\na^{k+2}u_1\\
&\quad-\int n_1{\rm div}\na^{k+2} u_1\na^{k+2}n_1 +n_1\na\na^{k+2} n_1\cdot\na^{k+2}u_1 \\
&\lesssim\left(\norm{\na n_1}_{L^\infty}\norm{\na^{k+2} u_1}_{L^2}+\norm{\na^{k+2}n_1}_{L^2}\norm{\na u_1}_{L^\infty}\right)\norm{\na^{k+2}n_1}_{L^2}\\
&\quad-\int n_1{\rm div}\Big(\na^{k+2} u_1\na^{k+2}n_1\Big) \\
&\lesssim\norm{\na (n_1,u_1)}_{L^\infty}\norm{\na^{k+2}(n_1,u_1)}_{L^2}^2+\int \na n_1\na^{k+2} u_1\na^{k+2}n_1 \\
&\lesssim\left(\norm{\na n_1}_{H^2}+\norm{u_1}_{H^3}\right) \norm{\na^{k+2}(n_1,u_1)}_{L^2}^2.\nonumber
\end{split}
\end{equation}
Applying the same arguments to these terms $I_{16}$--$I_{18}$, we deduce that for $l=k$ or $k+1$,
\begin{equation}
\begin{split}
I_{16}+I_{17}+I_{18}\lesssim\left(\norm{\na n_1}_{H^2}+\norm{(n_2,u_1,u_2)}_{H^3}\right)\left(\norm{\na^{l+1}(n_1,n_2)}_{L^2}^2+\norm{\na^l(u_1,u_2)}_{L^2}^2\right);\nonumber
\end{split}
\end{equation}
for $l=k+2$,
\begin{equation}
\begin{split}
I_{16}+I_{17}+I_{18}\lesssim\left(\norm{\na n_1}_{H^2}+\norm{(n_2,u_1,u_2)}_{H^3}\right)\norm{\na^{k+2}(n_1,n_2,u_1,u_2)}_{L^2}^2.\nonumber
\end{split}
\end{equation}

Hence, by these estimates for $I_{11}\sim I_{18}$, we deduce for $l=k,k+1$
\begin{equation}\label{Iyi kk+1}
\begin{split}
I_{1}\lesssim\left(\norm{\na n_1}_{H^2}+\norm{(n_2,u_1,u_2)}_{H^3}\right)\left(\norm{\na^{l+1}(n_1,n_2)}_{L^2}^2+\norm{\na^l(u_1,u_2)}_{L^2}^2\right);
\end{split}
\end{equation}
for $l=k+2$
\begin{equation}
\begin{split}
I_{1}\lesssim\left(\norm{\na n_1}_{H^2}+\norm{(n_2,u_1,u_2)}_{H^3}\right)\norm{\na^{k+2}(n_1,n_2,u_1,u_2)}_{L^2}^2.\nonumber
\end{split}
\end{equation}

Now we estimate the term $I_2$, and we must be much more careful about this term since the magnetic field $B$ has the weakest dissipative estimates. First of all, we have
\begin{eqnarray}\label{Ieryi}
I_2&&=\sum_{\ell=1}^{ l} C_l^\ell \int \na^{l-\ell} u_2\times \na^{\ell} B \cdot \na^l u_1+ \na^{l-\ell} u_1\times \na^{\ell} B \cdot \na^l u_2\nonumber\\
&&\lesssim C_{l}\sum_{\ell=1}^{l}\left(\norm{\na^{l-\ell}u_2 \na^{\ell}B}_{L^2}\norm{\na^l u_1}_{L^2}+\norm{\na^{l-\ell}u_1 \na^{\ell}B}_{L^2}\norm{\na^l u_2}_{L^2}\right).
\end{eqnarray}
Here we notice $\na^lu_2\times B \cdot \na^lu_1+ \na^lu_1\times B \cdot \na^lu_2=0$. We only estimate the first term on the right-hand side of \eqref{Ieryi}, the second term can be estimated similarly. We again have to distinct the arguments by the value of $l$. First, let $l=k$. We take $L^3-L^6$ and then apply Lemma \ref{A1} to have
\begin{equation}
\begin{split}
\norm{\na^{k-\ell} u_2\na^{\ell}B}_{L^{2}}&\lesssim \norm{\na^{k-\ell} u_2}_{L^3} \norm{\na^{\ell}B}_{L^6}
\\&\lesssim \norm{\na^\alpha u_2}^{\frac {\ell}{k}}_{L^2}\norm{\na^{k}u_2}_{L^2}^{1-\frac{\ell}{k}} \norm{\na B}_{L^2}^{1-\frac{\ell}{k}}\norm{\na^{k+1}B}_{L^2}^{\frac {\ell}{k}},\nonumber
 \end{split}
\end{equation}
where $\alpha$ is defined by
\begin{equation}
k-\ell+\frac12=\alpha\times\frac{\ell}{k}
+k\times \left(1-\frac{\ell}{k}\right)
 \Longrightarrow \alpha=\frac{ k }{2\ell}\le \frac{k}{2}.\nonumber
\end{equation}
Hence by Young's inequality, we have that for $l=k$,
\begin{equation}\label{I3 k}
I_2\le C_k\left( \norm{(u_1,u_2)}_{H^{\frac{k}{2}}}  +\norm{\na B}_{L^2}  \right) \left(\norm{\na^{k}(u_1,u_2)}_{L^2}^2 +\norm{\na^{k+1}B}_{L^2}^2 \right).
\end{equation}

We then let $l=k+1$. If $1\le \ell\le k$, we take $L^3-L^6$ and by  Lemma \ref{A1} again to obtain
\begin{equation}
\begin{split}
\norm{\na^{k+1-\ell} u_2\na^{\ell}B}_{L^{2}}&\lesssim \norm{\na^{k+1-\ell} u_2}_{L^3} \norm{\na^{\ell}B}_{L^6}
\\&\lesssim \norm{\na^\alpha u_2}^{\frac {\ell}{k}}_{L^2}\norm{\na^{k+1}u_2}_{L^2}^{1-\frac{\ell}{k}} \norm{\na B}_{L^2}^{1-\frac{\ell}{k}}\norm{\na^{k+1}B}_{L^2}^{\frac {\ell}{k}},\nonumber
\end{split}
\end{equation}
where $\alpha$ is defined by
\begin{equation}
k+1-\ell+\frac12=\alpha\times\frac {\ell}{k}
+(k+1)\times \left({1-\frac{\ell}{k}}\right)
\Longrightarrow \alpha=1+\frac{k}{2\ell}\le \frac{k}{2}+1;\nonumber
\end{equation}
if $\ell=k+1$, we take $L^\infty-L^2$ to get
\begin{equation}
\norm{ u_2\na^{k+1}B}_{L^{2}}\lesssim \norm{ u_2}_{L^\infty} \norm{\na^{k+1}B}_{L^2}.\nonumber
\end{equation}
We thus have that for $l=k+1$, by Sobolev's inequality,
\begin{equation}
I_2\le C_k\left( \norm{(u_1,u_2)}_{H^{\frac{k}{2}+1}\cap H^2}  +\norm{\na B}_{L^2}  \right) \left(\norm{\na^{k+1}(u_1,u_2)}_{L^2}^2 +\norm{\na^{k+1}B}_{L^2}^2 \right).\nonumber
\end{equation}

We now let $l=k+2$. If $1\le \ell\le k$, we take $L^3-L^6$ and by Lemma \ref{A1} again to have
\begin{equation}
\begin{split}
\norm{\na^{k+2-\ell} u_2\na^{\ell}B}_{L^{2}}&\lesssim \norm{\na^{k+2-\ell} u_2}_{L^3} \norm{\na^{\ell}B}_{L^6}
\\&\lesssim \norm{\na^\alpha u_2}^{\frac {\ell}{k}}_{L^2}\norm{\na^{k+2}u_2}_{L^2}^{1-\frac{\ell}{k}} \norm{\na B}_{L^2}^{1-\frac{\ell}{k}}\norm{\na^{k+1}B}_{L^2}^{\frac {\ell}{k}},\nonumber
\end{split}
\end{equation}
where $\alpha$ is defined by
\begin{equation}
k+2-\ell+\frac12=\alpha\times\frac {\ell}{k}
+(k+2)\times \left({1-\frac{\ell}{k}}\right)
\Longrightarrow \alpha=2+\frac{k}{2\ell}\le \frac{k}{2}+2;\nonumber
\end{equation}
if $\ell=k+1$ or $ k+2$, we take $L^\infty-L^2$ to get
\begin{equation}
\norm{ \na u_2\na^{k+1}B}_{L^{2}}\lesssim \norm{\na u_2}_{L^\infty} \norm{\na^{k+1}B}_{L^2},\nonumber
\end{equation}
and
\begin{equation}
\norm{ u_2\na^{k+2}B}_{L^{2}}\lesssim \norm{  u_2}_{L^\infty} \norm{\na^{k+2}B}_{L^2}.\nonumber
\end{equation}
We thus have that for $l=k+2$,
\begin{equation}
\begin{split}
I_2&\le C_k\left( \norm{(u_1,u_2)}_{H^{\frac{k}{2}+2}\cap H^3}  +\norm{\na B}_{L^2}  \right) \left(\norm{\na^{k+2}(u_1,u_2)}_{L^2}^2 +\norm{\na^{k+1}B}_{L^2}^2 \right)
\\&\quad+C\norm{(u_1,u_2)}_{L^\infty} \norm{\na^{k+2}B}_{L^2}\norm{\na^{k+2}(u_1,u_2)}_{L^2}.\nonumber
\end{split}
\end{equation}

We now estimate the last term $I_3$ in \eqref{I1-I3}. First, we split $I_3$ as:
\begin{eqnarray}\label{Isanyi}
I_{3}&&=\nu\sum_{\ell=0}^{ l}C_l^\ell \int\left[ \na^{\ell}f\left(\frac{n_1-n_2}{2}\right)\na^{l-\ell}\left(\frac{u_1-u_2}{2}\right)-\na^{\ell}f\left(\frac{n_1+n_2}{2}\right)\na^{l-\ell}\left(\frac{u_1+u_2}{2}\right)\right]\cdot\na^{l} E \nonumber\\
&&=\frac{\nu}{2}\sum_{\ell=0}^{ l}C_l^\ell \int\na^{\ell}f\left(\frac{n_1-n_2}{2}\right)\na^{l-\ell}u_1\cdot\na^{l} E-\frac{\nu}{2}\sum_{\ell=0}^{ l}C_l^\ell \int\na^{\ell}f\left(\frac{n_1-n_2}{2}\right)\na^{l-\ell}u_2\cdot\na^{l} E  \nonumber\\
&&\quad-\frac{\nu}{2}\sum_{\ell=0}^{ l}C_l^\ell \int\na^{\ell}f\left(\frac{n_1+n_2}{2}\right)\na^{l-\ell}u_1\cdot\na^{l} E+\frac{\nu}{2}\sum_{\ell=0}^{ l}C_l^\ell \int\na^{\ell}f\left(\frac{n_1+n_2}{2}\right)\na^{l-\ell}u_2\cdot\na^{l} E  \nonumber\\
&&:=\frac{\nu}{2}I_{31}+\frac{\nu}{2}I_{32}+\frac{\nu}{2}I_{33}+\frac{\nu}{2}I_{34}.
\end{eqnarray}
 We still have to distinct the arguments by the value of $l$. For $l=k$ or $k+1$,
we only estimate the first term $I_{31}$ on the right-hand side of \eqref{Isanyi}, the other terms $I_{32}$--$I_{34}$ can be estimated similarly.
If $0\le \ell\le l-1$, we take $L^\infty-L^2$ and by Lemma \ref{A1} and the estimate \eqref{fkn0} of Lemma \ref{A2} to obtain
\begin{equation}
\begin{split}
&\norm{\na^{\ell}f\left(\frac{n_1-n_2}{2}\right)\na^{l-\ell}u_1}_{L^2}\\
&\quad\le  \norm{\na^{\ell} f\left(\frac{n_1-n_2}{2}\right)}_{L^\infty} \norm{\na^{l-\ell}u_1}_{L^2}
\\&\quad\le C_l\norm{\na^{\ell+1} n_1}_{L^2}^{\frac12} \norm{\na^{\ell+2}n_1}_{L^2}^{\frac12}\norm{\na^{l-\ell}u_1}_{L^2}
+C_l\norm{\na^{\ell+1} n_2}_{L^2}^{\frac12} \norm{\na^{\ell+2}n_2}_{L^2}^{\frac12}\norm{\na^{l-\ell}u_1}_{L^2}
\\&\quad\le C_l \left(\norm{\na n_1}_{L^2}^{\frac{l-\ell}{l}}\norm{\na^{l+1} n_1}_{L^2}^{\frac{\ell}{l}}\right)^{\frac12}  \left(\norm{\na n_1}_{L^2}^{\frac{l-\ell-1}{l}} \norm{\na^{l+1} n_1}_{L^2}^{\frac{\ell+1}{l}}\right)^{\frac12}\norm{\na^{l-\ell}u_1}_{L^2}
\\&\quad\quad+ C_l \left(\norm{\na n_2}_{L^2}^{\frac{l-\ell}{l}}\norm{\na^{l+1} n_2}_{L^2}^{\frac{\ell}{l}}\right)^{\frac12}  \left(\norm{\na n_2}_{L^2}^{\frac{l-\ell-1}{l}} \norm{\na^{l+1} n_2}_{L^2}^{\frac{\ell+1}{l}}\right)^{\frac12}\norm{\na^{l-\ell}u_1}_{L^2}
\\&\quad\le C_l \norm{\na n_1}_{L^2}^{1-\frac{2\ell+1}{2l}}\norm{\na^{l+1} n_1}_{L^2}^{\frac{2\ell+1}{2l}} \norm{\na^{\alpha} u_1}_{L^2}^{\frac{2\ell+1}{2l}} \norm{\na^{l} u_1}_{L^2}^{1-\frac{2\ell+1}{2l}}
\\&\quad\quad+C_l \norm{\na n_2}_{L^2}^{1-\frac{2\ell+1}{2l}}\norm{\na^{l+1} n_2}_{L^2}^{\frac{2\ell+1}{2l}}\norm{\na^{\alpha} u_1}_{L^2}^{\frac{2\ell+1}{2l}} \norm{\na^{l} u_1}_{L^2}^{1-\frac{2\ell+1}{2l}},\nonumber
\end{split}
\end{equation}
where $\alpha$ is defined by
\begin{equation}
\begin{split}
l-\ell=\alpha\times\frac{2\ell+1}{2l}
+l\times\left(1-\frac{2\ell+1}{2l}\right)
\Longrightarrow \alpha=\frac{l}{2\ell+1}\le l;\nonumber
\end{split}
\end{equation}
if $\ell=l$, we take $L^2-L^\infty$ and by the estimate \eqref{fkn0} of Lemma \ref{A2} to have
\begin{equation}
\norm{\na^{l}f\left(\frac{n_1-n_2}{2}\right)u_1}_{L^2}\le \norm{\na^{l}f\left(\frac{n_1-n_2}{2}\right)}_{L^6}\norm{u_1}_{L^3}
\le C_l\norm{\na^{l+1} (n_1,n_2) }_{L^2}\norm{u_1}_{H^1}.\nonumber
\end{equation}
We thus have that for $l=k$ or $k+1$,
\begin{equation}
I_{31}\le C_l  \left(\norm{  \na (n_1,n_2)}_{L^{2}}+\norm{u_1}_{H^{l}\cap H^1}\right)  \left(\norm{\na^{l+1} (n_1,n_2)}_{L^2}^2+\norm{\na^l u_1}_{L^2}^2+\norm{\na^l E}_{L^2}^2 \right).\nonumber
\end{equation}
Hence, we have that for $l=k$ or $k+1$,
\begin{equation}\label{I4 k k+1}
I_3\le C_l  \left(\norm{  \na (n_1,n_2)}_{L^{2}}+\norm{(u_1,u_2)}_{H^{l}\cap H^1}\right) \left(\norm{\na^{l+1} (n_1,n_2)}_{L^2}^2+\norm{\na^l (u_1,u_2)}_{L^2}^2+\norm{\na^l E}_{L^2}^2 \right).
\end{equation}

Now for $l=k+2$, we rewrite $I_{31}+I_{33}$ as
\begin{equation}
\begin{split}
I_{31}+I_{33}&=\sum_{\ell=0}^{ k+2}C_{k+2}^{\ell}\int \na^{\ell}g\na^{k+2-\ell}u_1\cdot\na^{k+2} E \\
&=\int \left(g\na^{k+2}u_1+ \na^{k+2}gu_1\right)\cdot\na^{k+2} E-\sum_{\ell=1}^{ k+1} C_{k+2}^{\ell}\int \na\left(\na^{k+2-\ell}g\na^{\ell}u_1\right)\cdot\na^{k+1} E
\\
&=\int \left(g\na^{k+2}u_1+ \na^{k+2}gu_1\right)\cdot\na^{k+2} E-(k+2)\int\left(\na^{k+2}g\na u_1+\na g\na^{k+2}u_1\right)\cdot\na^{k+1} E
\\&\quad-\sum_{\ell=2}^{ k+1} C_{k+2}^{\ell}\int \na^{k+3-\ell}g\na^{\ell}u_1 \cdot\na^{k+1} E
-\sum_{\ell=1}^{ k} C_{k+2}^{\ell}\int  \na^{k+2-\ell}g\na^{\ell+1}u_1\cdot\na^{k+1} E
\\&:=I_{311} +I_{312} +I_{313}+I_{314},\nonumber
\end{split}
\end{equation}
where the function $g$ is defined as
\begin{equation}\label{gf}
g:=f\left(\frac{n_1-n_2}{2}\right)-f\left(\frac{n_1+n_2}{2}\right).
\end{equation}
By Lemma \ref{A2} and \eqref{fn}, we have
\begin{equation}
\begin{split}
I_{311} &\le C_k\left(\norm{g}_{L^\infty}\norm{\na^{k+2}u_1}_{L^2}+\norm{\na^{k+2}g}_{L^2}\norm{u_1}_{L^\infty}\right)\norm{\na^{k+2} E}_{L^2}
\\&\le C_k\norm{(n_2,u_1)}_{L^\infty}\norm{\na^{k+2}(n_2,u_1)}_{L^2}\norm{\na^{k+2} E}_{L^2}\nonumber
\end{split}
\end{equation}
and
\begin{equation}
\begin{split}
I_{312} &\le C_k\left(\norm{\na^{k+2} g}_{L^2}\norm{\na u_1}_{L^\infty}+\norm{\na g}_{L^\infty}\norm{\na^{k+2} u_1}_{L^2}\right)\norm{\na^{k+1} E}_{L^2}
\\&\le C_k\norm{\na(n_2,u_1)}_{L^\infty}\norm{\na^{k+2}(n_2,u_1)}_{L^2}\norm{\na^{k+1} E}_{L^2}.\nonumber
\end{split}
\end{equation}
As for the cases $l=k,k+1$ for $I_{3}$, we can bound $I_{313}$ and $I_{314}$ by
\begin{equation}
\begin{split}
& I_{313}+I_{314}\\
 &\quad\le C_k\left(\norm{  \na (n_1,n_2)}_{L^{2}}+\norm{u_1}_{H^{k+1}}\right)  \left(\norm{\na^{k+2}( n_1,n_2)}_{L^2}^2+\norm{\na^{k+1}u_1}_{L^2}^2 +\norm{\na^{k+1} E}_{L^2}^2 \right).\nonumber
\end{split}
\end{equation}
Hence, we have that for $l=k+2$,
\begin{equation}
\begin{split}
& I_{31}+I_{33}\\
&\quad\le C_k\left(\norm{\na n_1}_{L^2}+\norm{(n_2,u_1)}_{H^{k+1}\cap H^3}\right) \left(\norm{\na^{k+1}u_1}_{L^2}^2 +\norm{\na^{k+2}( n_1,n_2,u_1)}_{L^2}^2+\norm{\na^{k+1} E}_{L^2}^2 \right)
 \\&\qquad+C_k\norm{(n_2,u_1)}_{L^\infty}\norm{\na^{k+2}(n_2,u_1)}_{L^2}\norm{\na^{k+2} E}_{L^2}.\nonumber
 \end{split}
\end{equation}
Similarly, we can estimate $I_{32}+I_{34}$ for $l=k+2$. So, we have for $l=k+2$,
\begin{equation}
\begin{split}
 I_{3} &\le C_k\left(\norm{\na n_1}_{L^2}+\norm{(n_2,u_1,u_2)}_{H^{k+1}\cap H^3}\right) \left(\sum_{l=k+1}^{k+2}\norm{\na^{l}( n_1,n_2,u_1,u_2)}_{L^2}^2 +\norm{\na^{k+1} E}_{L^2}^2 \right)
 \\&\quad+C_k\norm{(n_2,u_1,u_2)}_{L^\infty}\norm{\na^{k+2}(n_2,u_1,u_2)}_{L^2}\norm{\na^{k+2} E}_{L^2}.\nonumber
 \end{split}
\end{equation}

Consequently, plugging these estimates for $I_1\sim I_3$ into \eqref{I1-I3} with $l=k,k+1,k+2$, and then summing up, we deduce \eqref{energy 1}.
\end{proof}
Note that in Lemma \ref{energy lemma} we only derive the dissipative estimates of $u_1$ and $u_2$. We now recover the dissipative
estimates of $ n_1,n_2, E$ and $B$ by constructing some interactive energy functionals in the following lemma.
\begin{lemma}\label{other di}
For any integer $k\ge0$, we have that for any small fixed $\eta>0$,
\begin{eqnarray}\label{other dissipation}
&&\frac{d}{dt}\left(\sum_{l=k}^{k+1}\int \na^lu_1\cdot\na\na^{l} n_1+\na^lu_2\cdot\na\na^{l} n_2-\sum_{l=k}^{k+1}\int \na^{l}u_2\cdot\na^lE-\eta\int
\na^k E \cdot\na^{k}\na\times B \right)\nonumber
\\&&\quad+\la \left(\sum_{l=k+1}^{k+2}\norm{\na^{l}n_1}_{L^2}^2+\sum_{l=k}^{k+2}\norm{\na^{l}n_2}_{L^2}^2+\sum_{l=k}^{k+1}\norm{\na^{l}E}_{L^2}^2+ \norm{\na^{k+1} B}_{L^2}^2\right)\nonumber
\\ &&\ \le
 C\sum_{l=k}^{k+2}\norm{\na^{l} (u_1,u_2)}_{L^2}^2\nonumber\\
 &&\quad+C_kG\left(\sum_{l=k+1}^{k+2}\norm{ \na^{l}n_1}_{L^2}^2+\sum_{l=k}^{k+2}\norm{ \na^{l}(n_2, u_1,u_2 )}_{L^2}^2+\norm{\na^{k+1}B}_{L^2}^2 \right),
\end{eqnarray}
where $G$ is defined by
$$G=G(n_1,n_2,u_1,u_2,B):=\norm{\na n_1}_{H^2}^2+\norm{ (n_2, u_1,u_2)}_{H^{k+1}\cap H^3}^2+\norm{\na B}_{L^2}^2.$$
\end{lemma}

\begin{proof} We divide the proof into four steps.

{\it Step 1: Dissipative estimates of $n_1, n_2$.}

Applying $\na^l$ ($l=k,k+1$) to $\eqref{yi}_2, \eqref{yi}_4$ and then taking the $L^2$ inner product with $\na\na^{l} n_1, \na\na^{l}n_2$ respectively, we obtain
\begin{eqnarray}  \label{n1n2yi}
\int&&\partial_t\na^lu_1 \cdot\na \na^{l} n_1 +  \partial_t\na^lu_2 \cdot\na \na^{l} n_2 +\norm{\na\na^{l } (n_1,n_2)}_{L^2}^2\nonumber\\
&&\le 2\nu \int  \na^{l}E\cdot\na \na^{l} n_2 +C \norm{\na^{l} u_1}_{L^2}\norm{\na^{l+1}  n_1}_{L^2}+C \norm{\na^{l} u_2}_{L^2}\norm{\na^{l+1}  n_2}_{L^2}\nonumber\\
&&\quad  +\norm{\na^{l}\left(u_1\cdot\na u_2+u_2\cdot\na u_1+n_1\na n_2+n_2\na n_1+u_1\times B\right)}_{L^2}\norm{\na^{l+1}  n_2}_{L^2}\nonumber\\
&&\quad  +\norm{\na^{l}\left(u_1\cdot\na u_1+u_2\cdot\na u_2+n_1\na n_1+n_2\na n_2+u_2\times B\right)}_{L^2}\norm{\na^{l+1}  n_1}_{L^2}.
\end{eqnarray}

The delicate first term on the left-hand side of \eqref{n1n2yi}
involves $\partial_t\na^{l} (u_1,u_2)$, and the key idea is to integrate by parts
in the $t$-variable and use the equations $\eqref{yi}_1$ and $\eqref{yi}_3$. Thus integrating by
parts for both the $t$- and $x$-variables, we obtain
\begin{equation}
\begin{split}
\int & \na^{l} \partial_tu_1\cdot\na\na^l n_1+ \na^{l} \partial_tu_2\cdot\na\na^l n_2\\
&=\frac{d}{dt}\int
\na^{l}u_1\cdot\na\na^l n_1+\na^{l}u_2\cdot\na\na^l n_2 -\int  \na^{l}
 u_1\cdot \na \na^{l}\partial_t n_1+\na^{l}
 u_2\cdot \na \na^{l}\partial_t n_2
 \\
&=\frac{d}{dt}\int
\na^{l}u_1\cdot\na\na^l n_1+\na^{l}u_2\cdot\na\na^l n_2 +\int  \na^{l}
{\rm div} u_1\na^{l}\partial_t n_1+\na^{l}
{\rm div} u_2\na^{l}\partial_t n_2  \\
& =\frac{d}{dt}\int
\na^{l}u_1\cdot\na\na^l n_1+\na^{l}u_2\cdot\na\na^l n_2 -\norm{\na^{l} {{\rm div} }(u_1,u_2)}_{L^2}^2+\int \na^{l} {\rm div} u_1\na^{l}g_1+\na^{l} {\rm div} u_2\na^{l}g_3
\\
& \ge \frac{d}{dt}\int
\na^{l}u_1\cdot\na\na^l n_1+\na^{l}u_2\cdot\na\na^l n_2 -C\norm{\na^{l+1} (u_1,u_2)}_{L^2}^2\\
 &\quad
 -C\norm{\na^{l}(u_1 \cdot \na n_1,u_2 \cdot \na n_1, n_1{\rm div} u_1,n_1{\rm div} u_2)}_{L^2}^2\\
 &\quad
 -C\norm{\na^{l}(u_1 \cdot \na n_2,u_2 \cdot \na n_2, n_2{\rm div} u_2,n_2 {\rm div} u_1)}_{L^2}^2.\nonumber
\end{split}
\end{equation}
First, we have
\begin{equation}
\begin{split}
\norm{\na^{l}(u_1\cdot\na n_1)}_{L^2}\lesssim\sum_{0\le\ell\le l}\norm{\na^{\ell}u_1\cdot\na\na^{l-\ell}n_1}_{L^2}.\nonumber
\end{split}
\end{equation}
If $\ell=0$, then
\begin{equation}  \label{n1n2si}
\begin{split}
\norm{u_1\cdot\na\na^{l}n_1 }_{L^2}\lesssim\norm{u_1}_{L^\infty}\norm{\na^{l+1}n_1}_{L^2}\lesssim\norm{u_1}_{H^2}\norm{\na^{l+1}n_1}_{L^2};
\end{split}
\end{equation}
if $1\le \ell\le \left[\frac{l}{2}\right]$, by H\"older's inequality and Lemma A.1, we have
\begin{eqnarray}\label{n1n2wu}
\norm{\na^{\ell}u_1\cdot\na\na^{l-\ell}n_1}_{L^2}&&\lesssim \norm{\na^{l+1-\ell} n_1}_{L^6} \norm{\na^{\ell}u_1}_{L^3}\nonumber
\\&&\lesssim \norm{\na n_1}_{L^2}^{\frac{\ell-1}{l}}  \norm{\na^{l+1} n_1}_{L^2}^{\frac{l-\ell+1}{l}} \norm{\na^\alpha u_1}_{L^2}^{\frac{l-\ell+1}{l}}\norm{\na^{l+1}u_1}_{L^2}^{\frac{\ell-1}{l}}\nonumber
\\&&\lesssim\left(\norm{\na n_1}_{L^2}+\norm{u_1}_{H^3}\right)\norm{\na^{l+1}(n_1,u_1)}_{L^2},
\end{eqnarray}
where $\alpha$ is defined by
\begin{equation}
\begin{split}
\ell+\frac{1}{2}=\alpha\times\frac{l-\ell+1}{l}
+(l+1)\times\frac{\ell-1}{l} \Longrightarrow \alpha=\frac{3l-2\ell+2}{2l-2\ell+2}\in \left[\frac{3}{2},3\right)\, \text{ since }\ell\le \frac{l}{2};\nonumber
\end{split}
\end{equation}
if $\left[\frac l2\right]+1\le \ell\le l$, by H\"older's inequality and Lemma A.1 again, we have
\begin{eqnarray}\label{n1n2qi}
\norm{\na^{\ell}u_1\cdot\na\na^{l-\ell}n_1}_{L^2}&&\lesssim \norm{\na^{l+1-\ell} n_1}_{L^3} \norm{\na^{\ell}u_1}_{L^6}\nonumber
\\&&\lesssim \norm{\na^\alpha n_1}_{L^2}^{\frac{\ell+1}{l+1}}  \norm{\na^{l+1} n_1}_{L^2}^{\frac{l-\ell}{l+1}} \norm{ u_1}_{L^2}^{\frac{l-\ell}{l+1}}\norm{\na^{l+1}u_1}_{L^2}^{\frac{\ell+1}{l+1}}\nonumber
\\&&\lesssim\left(\norm{\na n_1}_{H^2}+\norm{u_1}_{H^3}\right)\norm{\na^{l+1}(n_1,u_1)}_{L^2},
\end{eqnarray}
where $\alpha$ is defined by
\begin{equation}
\begin{split}
l-\ell+\frac{3}{2}=\alpha\times\frac{\ell+1}{l+1}
+l-\ell
 \Longrightarrow \alpha=\frac{3l+3}{2\ell+2}\in \left[\frac{3}{2},3\right)\, \text{ since }\ell\ge \frac{l+1}{2}.\nonumber
\end{split}
\end{equation}
Hence, by \eqref{n1n2si}--\eqref{n1n2qi}, we have
\begin{equation}\label{n1n2jiu}
\norm{\na^{l}(u_1\cdot\na n_1)}_{L^2}\lesssim\left(\norm{\na n_1}_{H^2}+\norm{u_1}_{H^3}\right)\norm{\na^{l+1}(n_1,u_1)}_{L^2}.
\end{equation}
Similarly, we also have
\begin{equation}
\norm{\na^{l}(u_2 \cdot \na n_1, n_1{\rm div} u_1,n_1{\rm div} u_2)}_{L^2}\lesssim\left(\norm{\na n_1}_{H^2}+\norm{(u_1,u_2)}_{H^3}\right)\norm{\na^{l+1}(n_1,u_1,u_2)}_{L^2}.
\end{equation}
Using the commutator estimate of Lemma \ref{commutator}, we have
\begin{eqnarray}
\norm{\na^{l}(u_1 \cdot \na n_2)}_{L^2}&&\le \norm{ u_1 \cdot \na^{l} \na n_2 }_{L^2}+\norm{\left[\na^l ,u_1 \right]\cdot \na n_2 }_{L^2}\nonumber
\\&&\le  \norm{ u_1}_{L^\infty} \norm{ \na^{l+1}   n_2 }_{L^2}+C_l\norm{\na u_1}_{L^\infty}
\norm{\na^{l}n_2}_{L^2}+C_l\norm{\na^l u_1}_{L^2}\norm{\na n_2}_{L^\infty}\nonumber
\\&&\le C_l \norm{(n_2, u_1)}_{H^3} \left(\norm{ \na^{l}  ( n_2,u_1) }_{L^2}+\norm{ \na^{l+1}   n_2 }_{L^2}\right).
\end{eqnarray}
Similarly,
\begin{eqnarray}\label{n1n2shier}
&&\norm{\na^{l}(u_2\cdot\na n_2, n_2 {\rm div} u_2,n_2{\rm div}u_1)}_{L^2}\nonumber\\
&&\quad\le C_l \norm{(n_2,u_1, u_2)}_{H^3} \left(\norm{ \na^{l}  ( n_2,u_1,u_2) }_{L^2}+\norm{ \na^{l+1} (n_2,u_1,u_2) }_{L^2}\right).
\end{eqnarray}
Hence, we obtain
\begin{eqnarray}  \label{n1n2shisan}
&&\int\na^{l} \partial_tu_1\cdot\na\na^l n_1+ \na^{l} \partial_tu_2\cdot\na\na^l n_2\nonumber\\
&&\ge \frac{d}{dt}\int
\na^{l}u_1\cdot\na^l\na n_1+\na^{l}u_2\cdot\na^l\na n_2 -C\norm{\na^{l+1}(u_1,u_2)}_{L^2}^2\nonumber
\\&&\ -C_l \left(\norm{\na n_1}_{H^2}^2+\norm{(n_2,u_1,u_2)}_{H^3}^2\right) \left(\norm{ \na^{l}  ( n_2,u_1,u_2) }_{L^2}^2+\norm{ \na^{l+1}  (n_1,n_2,u_1,u_2) }_{L^2}^2\right).
\end{eqnarray}

Next, integrating by parts and using the equation $\eqref{yi}_7$, we have
\begin{eqnarray}\label{n1n2shisi}
2\nu &&\int  \na^{l}E\cdot \na\na^{l} n_2 \nonumber\\
&&=-2\nu \int  \na^{l}{\rm div} E\na^{l} n_2
=-2\nu ^2\int  \na^{l}\left(f(\frac{n_1+n_2}{2})-f(\frac{n_1-n_2}{2})\right)\na^{l} n_2 \nonumber\\&&=-2\nu ^2\int  \na^{l}\left[n_2+f(\frac{n_1+n_2}{2})-f(\frac{n_1-n_2}{2})-n_2\right]\na^{l} n_2 \nonumber\\
&&\lesssim -\norm{\na^{l}n_2}_{L^2}^2+ \left(\norm{\na n_1}_{H^2}+\norm{n_2}_{H^3}\right)\left(\norm{\na^{l+1}(n_1,n_2)}_{L^2}+\norm{\na^{l}n_2}_{L^2}\right).
\end{eqnarray}
Here we have used the estimate
\begin{eqnarray}\label{n12}
&&\norm{\na^{l}\left[f(\frac{n_1+n_2}{2})-f(\frac{n_1-n_2}{2})-n_2\right]}_{L^2}\nonumber\\
&&\quad\lesssim  \left(\norm{\na n_1}_{H^2}+\norm{n_2}_{H^3}\right)\left(\norm{\na^{l+1}(n_1,n_2)}_{L^2}+\norm{\na^{l}n_2}_{L^2}\right).
\end{eqnarray}
In fact, by noticing that $f(\frac{n_1+n_2}{2})-f(\frac{n_1-n_2}{2})-n_2\sim n_1n_2$ and Lemma \ref{A2}, we have
\begin{equation}\label{n12 yi}
\begin{split}
\norm{\na^{l}\left[f(\frac{n_1+n_2}{2})-f(\frac{n_1-n_2}{2})-n_2\right]}_{L^2}
\lesssim\norm{\na^l(n_1n_2)}_{L^2}\lesssim  \sum_{0\le\ell\le l}\norm{\na^\ell n_1\na^{l-\ell}n_2}_{L^2}.
\end{split}
\end{equation}
If $\ell=0$, then
\begin{equation}
\begin{split}
\norm{n_1\na^ln_2}_{L^2}\lesssim\norm{n_1}_{L^6}\norm{\na^ln_2}_{L^3}\lesssim\norm{\na n_1}_{L^2}\norm{\na^ln_2}_{H^1};
\end{split}
\end{equation}
if $1\le\ell\le\left[\frac l2\right]$, by H\"older's inequality and Lemma A.1, we have
\begin{eqnarray}\label{n12 san}
\norm{\na^\ell n_1\na^{l-\ell}n_2}_{L^2}&&\lesssim\norm{\na^\ell n_1}_{L^3}\norm{\na^{l-\ell}n_2}_{L^6}\nonumber\\
&&\lesssim\norm{\na^\alpha n_1}_{L^2}^{\frac{l-\ell+1}{l}}\norm{\na^{l+1}n_1}_{L^2}^{\frac{\ell-1}{l}}\norm{n_2}_{L^2}^{\frac{\ell-1}{l}}\norm{\na^l n_2}_{L^2}^{\frac{l-\ell+1}{l}}\nonumber\\
&&\lesssim\left(\norm{\na n_1}_{H^2}+\norm{n_2}_{L^2}\right)\left(\norm{\na^{l+1}n_1}_{L^2}+\norm{\na^l n_2}_{L^2}\right),
\end{eqnarray}
where $\alpha$ is defined by
\begin{equation}
\begin{split}
\ell+\frac{1}{2}=\alpha\times\frac{l-\ell+1}{l}
+(l+1)\times \frac{\ell-1}{l}
 \Longrightarrow \alpha=\frac{3l-2\ell+2}{2l-2\ell+2}\in \left[\frac{3}{2},3\right)\, \text{ since }\ell\le \frac{l}{2};\nonumber
\end{split}
\end{equation}
if $\left[\frac l2\right]+1\le\ell\le l$, by H\"older's inequality and Lemma A.1 again, we have
\begin{eqnarray}\label{n12 wu}
\norm{\na^\ell n_1\na^{l-\ell}n_2}_{L^2}&&\lesssim\norm{\na^\ell n_1}_{L^6}\norm{\na^{l-\ell}n_2}_{L^3}\nonumber\\
&&\lesssim\norm{\na n_1}_{L^2}^{\frac{l-\ell}{l}}\norm{\na^{l+1}n_1}_{L^2}^{\frac{\ell}{l}}\norm{\na^\alpha n_2}_{L^2}^{\frac{\ell}{l}}\norm{\na^l n_2}_{L^2}^{\frac{l-\ell}{l}}\nonumber\\
&&\lesssim\left(\norm{\na n_1}_{L^2}+\norm{n_2}_{H^3}\right)\left(\norm{\na^{l+1}n_1}_{L^2}+\norm{\na^l n_2}_{L^2}\right),
\end{eqnarray}
where $\alpha$ is defined by
\begin{equation}
\begin{split}
l-\ell+\frac{1}{2}=\alpha\times\frac{\ell}{l}
+l-\ell
 \Longrightarrow \alpha=\frac{l}{2\ell}\in \left[\frac{1}{2},3\right)\, \text{ since }\ell\ge \frac{l+1}{2}.\nonumber
\end{split}
\end{equation}
By \eqref{n12 yi}--\eqref{n12 wu}, we complete the proof of \eqref{n12}.

Lastly, as in \eqref{n1n2jiu}--\eqref{n1n2shier}, we have
\begin{eqnarray} \label{n1n2shiwu}
&&\norm{\na^{l}\left(u_1\cdot\na u_2+u_2\cdot\na u_1+n_1\na n_2+n_2\na n_1\right)}_{L^2}\nonumber\\
&&\quad+\norm{\na^{l}\left(u_1\cdot\na u_1+u_2\cdot\na u_2+n_1\na n_1+n_2\na n_2\right)}_{L^2}\nonumber\\
&&\le C_l \left(\norm{\na n_1}_{H^2}+\norm{(n_2,u_1,u_2)}_{H^3}\right)  \left(\norm{ \na^{l}( n_2,u_1,u_2) }_{L^2}+\norm{ \na^{l+1} (n_1,n_2,u_1,u_2) }_{L^2}\right).
\end{eqnarray}
From the estimate of $I_2$ in Lemma \ref{energy lemma}, we have that for $l=k$ or $k+1$,
\begin{eqnarray} \label{n1n2shiliu}
&&\norm{\na^{l}\left(u_1\times B,u_2\times B\right)}_{L^2} \nonumber\\
&&\quad\le C_k\left( \norm{(u_1,u_2)}_{H^{\frac{k}{2}+1}\cap H^2}  +\norm{\na B}_{L^2}  \right) \left(\norm{\na^l (u_1,u_2)}_{L^2} +\norm{\na^{k+1}B}_{L^2} \right).
\end{eqnarray}

Plugging these estimates \eqref{n1n2shisan}, \eqref{n1n2shisi},\eqref{n1n2shiwu} and \eqref{n1n2shiliu} into
\eqref{n1n2yi}, by Cauchy's inequality, we obtain
\begin{eqnarray}  \label{n estimate}
\frac{d}{dt}&&\sum_{l=k}^{k+1}\int \na^lu_1\cdot\na\na^{l} n_1+\na^lu_2\cdot\na\na^{l} n_2+\la \left(\sum_{l=k+1}^{k+2}\norm{\na^{l}n_1}_{L^2}^2+\sum_{l=k}^{k+2}\norm{\na^{l}n_2}_{L^2}^2\right)\nonumber
\\ &&\le
 C\sum_{l=k}^{k+2}\norm{\na^{l}(u_1,u_2)}_{L^2}^2\nonumber\\
 &&\quad+C_kG\left(\sum_{l=k+1}^{k+2}\norm{ \na^{l}n_1}_{L^2}^2+\sum_{l=k}^{k+2}\norm{ \na^{l}(n_2, u_1,u_2 )}_{L^2}^2+\norm{\na^{k+1}B}_{L^2}^2 \right).
\end{eqnarray}
Here $G$ is well-defined above. This completes the dissipative estimates for $n_1,n_2$.

{\it Step 2: Dissipative estimate of $E$.}

Applying $\na^l$ ($l=k,k+1$) to $\eqref{yi}_4$ and then taking the $L^2$ inner product with $-\na^{l} E$, we obtain
\begin{eqnarray}  \label{Eyi}
-&&\int  \na^l \partial_tu_2 \cdot\na^{l}E +2\nu \norm{\na^{l} E}_{L^2}^2\le\int \na\na^{l} n_2\cdot \na^{l}E +C \norm{\na^{l} (u_1,u_2)}_{L^2}\norm{\na^{l}  E}_{L^2}\nonumber\\
 &&\qquad\qquad +\norm{\na^{l}\left(u_1\cdot\na
u_2+u_2\cdot\na u_1+ n_1\na n_2+n_2\na n_1+u_1\times B\right)}_{L^2}\norm{\na^{l } E}_{L^2}.
\end{eqnarray}

Again, the delicate first term on the left-hand side of \eqref{Eyi}
involves $\partial_t\na^{l} u_2$, and the key idea is to integrate by parts
in the $t$-variable and use the equation $\eqref{yi}_5$ in the Maxwell system. Thus we obtain
\begin{eqnarray}  \label{Eer}
 -&&\int  \na^{l} \partial_tu_2\cdot\na^lE \nonumber\\
 &&=-\frac{d}{dt}\int
\na^{l}u_2\cdot\na^lE +\int  \na^{l}
u_2\cdot\na^{l}\partial_t E  \nonumber\\
&&=-\frac{d}{dt}\int
\na^{l}u_2\cdot\na^lE -\nu \norm{\na^{l}u_2}_{L^2}^2+ \int \na^{l} u_2\cdot\na^{l}\left(g_5+
\nu\na\times B\right) .
\end{eqnarray}
From the estimates of $I_3$ in Lemma \ref{energy lemma}, we have that
\begin{equation}
\norm{\na^{l}g_5}_{L^2}\le C_l  \left(\norm{\na n_1}_{L^2}+\norm{(n_2,u_1,u_2)}_{H^{l}\cap H^1}\right)\left(\norm{\na^{l+1} (n_1,n_2)}_{L^2}+\norm{\na^l (u_1,u_2)}_{L^2} \right).\nonumber
\end{equation}
We must be much more careful about the remaining term in \eqref{Eer} since there is no small factor in front of it. The key is to use Cauchy's inequality and distinct the cases of $l=k$ and $l=k+1$ due to the weakest dissipative estimate of $B$. For $l=k$, we have
\begin{equation}
\nu \int \na^{k} u_2\cdot
\na\times \na^{k} B   \le \varepsilon \norm{ \na^{k+1} B}_{L^2}^2+C_\varepsilon \norm{\na^{k} u_2}_{L^2}^2;\nonumber
\end{equation}
for $l=k+1$, integrating by parts, we obtain
\begin{eqnarray}\label{Ewu}
\nu \int \na^{k+1} u_2\cdot
\na\times \na^{k+1} B  &&=
\nu \int \na\times \na^{k+1} u_2\cdot
 \na^{k+1} B\nonumber\\
&&\le \varepsilon \norm{ \na^{k+1} B}_{L^2}^2+C_\varepsilon \norm{\na^{k+2} u_2}_{L^2}^2.
\end{eqnarray}

Plugging these estimates \eqref{Eer}--\eqref{Ewu} and \eqref{n1n2shisi}, \eqref{n1n2shiwu} and \eqref{n1n2shiliu} from Step 1 into
\eqref{Eyi}, by Cauchy's inequality, we then obtain
\begin{eqnarray}  \label{E estimate}
&& -\frac{d}{dt}\sum_{l=k}^{k+1}\int \na^{l}u_2\cdot\na^lE +\la \sum_{l=k}^{k+1}\norm{\na^{l}E}_{L^2}^2
\le \varepsilon \norm{ \na^{k+1} B}_{L^2}^2+
 C_\varepsilon \sum_{l=k}^{k+2}\norm{\na^{l} (u_1,u_2)}_{L^2}^2\nonumber\\
&&\qquad\qquad+C_kG\left(\sum_{l=k+1}^{k+2}\norm{ \na^{l}n_1}_{L^2}^2+\sum_{l=k}^{k+2}\norm{ \na^{l}(n_2,u_1, u_2 )}_{L^2}^2+\norm{\na^{k+1}B}_{L^2}^2 \right).
\end{eqnarray}
This completes the dissipative estimate for $E$.

{\it Step 3: Dissipative estimate of $B$.}

Applying $\na^k$ to $\eqref{yi}_5$ and then taking the $L^2$ inner product with $-\na\times\na^{k} B$,  we obtain
\begin{eqnarray}  \label{Byi}
&&-\int \na^k \partial_tE \cdot\na\times\na^{k} B +\nu \norm{\na\times\na^{k} B}_{L^2}^2\nonumber
\\&&\quad\le \nu \norm{\na^{k}u_2}_{L^2}\norm{\na\times\na^{k} B}_{L^2}+
 \norm{\na^{k}g_5}_{L^2}\norm{\na\times\na^{k} B}_{L^2}.
\end{eqnarray}
Integrating by parts for both the $t$- and $x$-variables and using the equation $\eqref{yi}_6$, we have
\begin{eqnarray*}
-\int \na^k \partial_tE \cdot\na\times\na^{k} B
&&=-\frac{d}{dt}\int
\na^k E \cdot\na\times\na^{k} B +\int  \na\times\na^{k}
E\cdot\na^{k}\partial_t B  \nonumber\\
&&=-\frac{d}{dt}\int
\na^k E \cdot\na\times\na^{k} B -\nu \norm{\na\times\na^{k}
E}_{L^2}^2.
\end{eqnarray*}
From the estimates of $I_3$ in Lemma \ref{energy lemma}, we have that
\begin{equation*}
\norm{\na^{k}g_5}_{L^2}\le C_l  \left(\norm{\na n_1}_{L^2}+\norm{(n_2,u_1,u_2)}_{H^{k}\cap H^1}\right)\left(\norm{\na^{k+1} (n_1,n_2)}_{L^2}+\norm{\na^k (u_1,u_2)}_{L^2} \right).
\end{equation*}
Plugging the estimates above into \eqref{Byi} and by Cauchy's inequality, since ${\rm div} B=0$, we
then obtain
\begin{eqnarray}   \label{B estimate}
&&-\frac{d}{dt}\int
\na^k E \cdot\na^{k}\na\times B +\la \norm{\na^{k+1} B}_{L^2}^2\le C\norm{\na^{k}u_2}_{L^2}^2+C\norm{\na^{k+1}
  E}_{L^2}^2\nonumber\\&&\quad+
C_k  \left(\norm{\na n_1}_{L^2}^2+\norm{(n_2,u_1,u_2)}_{H^{k}\cap H^1}^2\right)\left(\norm{\na^{k+1} (n_1,n_2)}_{L^2}^2+\norm{\na^k (u_1,u_2)}_{L^2}^2 \right).
\end{eqnarray}
This completes the dissipative estimate for $B$.

{\it Step 4: Conclusion.}

Multiplying \eqref{B estimate} by a small enough but fixed constant $\eta$ and then adding it to \eqref{E estimate}  so that the second term on the right-hand side of \eqref{B estimate} can be absorbed, then choosing $\varepsilon$ small enough so that the first term on the right-hand side of \eqref{E estimate} can be absorbed; we obtain
\begin{equation}
\begin{split}
& \frac{d}{dt}\left(\sum_{l=k}^{k+1}\int \na^{l}u_2\cdot\na^lE-\eta\int
\na^k E \cdot\na^{k}\na\times B \right)+\la \left(\sum_{l=k}^{k+1}\norm{\na^{l}E}_{L^2}^2+ \norm{\na^{k+1} B}_{L^2}^2\right)
\\ &\quad\le  C \sum_{l=k}^{k+2}\norm{\na^{l} u_2}_{L^2}^2
 +C_kG\left(\sum_{l=k+1}^{k+2}\norm{ \na^{l}n_1}_{L^2}^2+\sum_{l=k}^{k+2}\norm{ \na^{l}(n_2, u_1,u_2)}_{L^2}^2+\norm{\na^{k+1}B}_{L^2}^2\right).\nonumber
\end{split}
\end{equation}
Adding the inequality above to \eqref{n estimate}, we get \eqref{other dissipation}.
\end{proof}

\subsection{Negative Sobolev estimates}

In this subsection, we will derive the evolution of the negative Sobolev
norms of $U:=(n_1,n_2,u_1,u_2,E,B)$. In order to estimate the nonlinear terms, we need to
restrict ourselves to that $s\in (0,3/2)$. We will establish the following lemma.

\begin{lemma}\label{negative Sobolev norm}
For $s\in(0,1/2]$, we have
\begin{equation}  \label{1E_s}
\begin{split}
\frac{d}{dt}\norm{U}_{\dot{H}^{-s}}^2 +\la\norm{(u_1,u_2)}_{\dot{H}^{-s}}^2\lesssim \left(\norm{(n_2,u_1,u_2)}_{H^2}^2+\norm{\na(n_1,
B)}_{H^1}^2 \right)\norm{U}_{\dot{H}^{-s}};
\end{split}
\end{equation}
and for $s\in(1/2,3/2)$, we have
\begin{eqnarray}  \label{1E_s2}
&&\frac{d}{dt}\norm{U}_{\dot{H}^{-s}}^2 +\la\norm{(u_1,u_2)}_{\dot{H}^{-s}}^2\lesssim\norm{(\na n_1,n_2,u_1,u_2)}_{H^{1}}^{2}\norm{U}_{\dot{H}^{-s}}\nonumber\\
 &&\quad+\norm{(n_1,B)}_{L^2}^{s-1/2}\norm{\na (n_1,B)}_{L^2}^{3/2-s}\norm{(\na n_1,\na n_2,u_1,u_2,\na u_1,\na u_2)}_{L^2}\norm{U}_{\dot{H}^{-s}}.
\end{eqnarray}
\end{lemma}

\begin{proof}
The $\Lambda^{-s}$ $(s>0)$ energy estimate of $\eqref{yi}_{1}$--$\eqref{yi}_{6}$ yield
\begin{eqnarray}  \label{1E_s_0}
&&\frac{d}{dt}\left(\frac{1}{2}\norm{(n_1,n_2,u_1,u_2)}_{\dot{H}^{-s}}^2+\norm{(E,B)}_{\dot{H}^{-s}}^2\right)  +\nu\norm{(u_1,u_2)}_{\dot{H}^{-s}}^2\nonumber
\\&&\quad=\int\Lambda^{-s}g_1\cdot\Lambda^{-s} n_1+\int \Lambda^{-s}\left(g_2+u_2\times B\right)
\cdot\Lambda^{-s} u_1 \nonumber\\
&&\qquad+\int\Lambda^{-s}g_3\cdot\Lambda^{-s} n_2+\int \Lambda^{-s}\left(g_4+u_1\times B\right)
\cdot\Lambda^{-s} u_2 +2\int\Lambda^{-s}g_5\cdot\Lambda^{-s} E\nonumber\\
&&\quad\lesssim\norm{g_1}_{\dot{H}^{-s}} \norm{n_1}_{\dot{H}^{-s}}+\norm{g_2+u_2\times B}_{\dot{H}^{-s}} \norm{u_1}_{\dot{H}^{-s}}\nonumber\\
&&\qquad+\norm{g_3}_{\dot{H}^{-s}} \norm{n_2}_{\dot{H}^{-s}}+\norm{ g_4+u_1\times B}_{\dot{H}^{-s}} \norm{u_2}_{\dot{H}^{-s}}+\norm{g_5}_{\dot{H}^{-s}} \norm{E}_{\dot{H}^{-s}}.
\end{eqnarray}

We now restrict the value of $s$ in order to estimate the other terms on
the right-hand side of \eqref{1E_s_0}. If $s\in (0,1/2]$, then $1/2+s/3<1$
and $3/s\geq 6$. Then applying Lemma \ref{Riesz lemma}, together with H\"{o}lder's, Sobolev's and Young's inequalities, we
obtain
\begin{equation}
\begin{split}
\norm{ u_1\cdot\na u_2 }
_{\dot{H}^{-s}}
&  \lesssim \norm{ u_1\cdot\na u_2 }_{L^{\frac{1}{1/2+s/3}}}
 \lesssim \norm{u_1}_{L^{3/s}}
\norm{\na u_2 }_{L^{2}} \\
&\lesssim \norm{\na u_1}_{L^{2}}^{1/2+s}
\norm{\na^2u_1 }_{L^{2}}^{1/2-s}\norm{\na u_2}_{L^2}  \\
& \lesssim \norm{ \na u_1}_{H^{1}}^{2}+\norm{\na u_2}_{L^{2}}^{2}.\nonumber
\end{split}
\end{equation}
We can similarly bound the other terms in the $g_1\sim g_5$ and $(u_1+u_2)\times B$. So we have
\begin{equation}\label{sum yi}
\begin{split}
&\sum_{i=1}^5\norm{g_i}_{\dot{H}^{-s}}+\norm{(u_1+u_2)\times B}_{\dot{H}^{-s}}\lesssim\norm{(n_2,u_1,u_2)}_{H^2}^2+\norm{\na(n_1,
B)}_{H^1}^2.
\end{split}
\end{equation}

Now if $s\in (1/2,3/2)$, we shall estimate the right-hand side of
\eqref{1E_s_0} in a different way. Since $s\in
(1/2,3/2)$, we have that $1/2+s/3<1$ and $2<3/s<6$. Then applying Lemma \ref{Riesz lemma} and using
(different) Sobolev's inequality, we have
\begin{equation}
\begin{split}
\norm{ u_1\cdot\na u_2 }_{\dot{H}^{-s}} &\lesssim \norm{ u_1}_{L^{3/s}}\norm{\na u_2 }_{L^{2}} \lesssim \norm{u_1}_{L^{2}}^{s-1/2}
\norm{\na u_1 }_{L^{2}}^{3/2-s}\norm{\na u_2}_{L^2}  \\
&\lesssim  \norm{u_1}_{H^{1}}^{2}+\norm{\na u_2}_{L^{2}}^{2}.\nonumber
\end{split}
\end{equation}
In particular, we must be careful about the terms involved with $n_1$ and $B$ since they are both degenerately dissipative. For example,
\begin{equation}
\begin{split}
&\norm{  n_1\na n_2 }_{\dot{H}^{-s}}\lesssim\norm{n_1}
_{L^{2}}^{s-1/2}\norm{ \na  n_1}_{L^{2}}^{3/2-s}\norm{\na n_2}
_{L^{2}};\\
&\norm{  u_2\times B}_{\dot{H}^{-s}}\lesssim\norm{B}
_{L^{2}}^{s-1/2}\norm{ \na B}_{L^{2}}^{3/2-s}\norm{u_2}
_{L^{2}}.\nonumber
\end{split}
\end{equation}
Then, we have
\begin{eqnarray}\label{sum er}
&&\sum_{i=1}^5\norm{g_i}_{\dot{H}^{-s}}+\norm{(u_1+u_2)\times B}_{\dot{H}^{-s}}\lesssim \norm{(\na n_1,n_2,u_1,u_2)}_{H^{1}}^{2}\nonumber\\
&&\qquad\qquad+ \norm{(n_1,B)}
_{L^{2}}^{s-1/2}\norm{ \na (n_1,B)}_{L^{2}}^{3/2-s}\norm{(\na n_1,\na n_2, u_1,u_2,\na u_1,\na u_2)}_{L^{2}}.
\end{eqnarray}
Hence, we deduce \eqref{1E_s} from \eqref{sum yi} and \eqref{1E_s2} from \eqref{sum er}.
\end{proof}

\subsection{Negative Besov estimates}

In this subsection, we will derive the evolution of the negative Besov
norms of $U:=(n_1,n_2,u_1,u_2,E,B)$. The argument is similar to the previous subsection.

\begin{lemma}\label{negative Besov norm}
For $s\in(0,1/2]$, we have
\begin{equation}
\begin{split}
\frac{d}{dt}\norm{U}_{\dot{B}_{2,\infty}^{-s}}^2 +\la\norm{(u_1,u_2)}_{\dot{B}_{2,\infty}^{-s}}^2\lesssim \left(\norm{(n_2,u_1,u_2)}_{H^2}^2+\norm{\na(n_1,
B)}_{H^1}^2 \right)\norm{U}_{\dot{B}_{2,\infty}^{-s}};\nonumber
\end{split}
\end{equation}
and for $s\in(1/2,3/2]$, we have
\begin{equation}
\begin{split}
&\frac{d}{dt}\norm{U}_{\dot{B}_{2,\infty}^{-s}}^2 +\la\norm{(u_1,u_2)}_{\dot{B}_{2,\infty}^{-s}}^2\lesssim\norm{(\na n_1,n_2,u_1,u_2)}_{H^{1}}^{2}\norm{U}_{\dot{B}_{2,\infty}^{-s}}\\
 &\quad+\norm{(n_1,B)}_{L^2}^{s-1/2}\norm{\na (n_1,B)}_{L^2}^{3/2-s}\norm{(\na n_1,\na n_2,u_1,u_2,\na u_1,\na u_2)}_{L^2}\norm{U}_{\dot{B}_{2,\infty}^{-s}}.\nonumber
\end{split}
\end{equation}
\end{lemma}

\begin{proof}
The $\dot{\Delta}_{j}$ energy estimates of $\eqref{yi}_1$--$\eqref{yi}_6$ yield, with multiplication of $2^{-2sj}$ and then taking the supremum over ${j\in\mathbb{Z}}$,
\begin{equation}
\begin{split}
&\frac{d}{dt}\left(\frac12\norm{(n_1,n_2,u_1,u_2)}_{\dot{B}_{2,\infty}^{-s}}^2+\norm{(E,B)}_{\dot{B}_{2,\infty}^{-s}}^2\right) +\nu \norm{(u_1,u_2)}_{\dot{B}_{2,\infty}^{-s}}^2\\
&\lesssim\sup\limits_{j\in\mathbb{Z}}2^{-2sj}\int  \dot{\Delta}_{j}g_1
\cdot\dot{\Delta}_{j} n_1 + \dot{\Delta}_{j}\left(g_2+u_2\times B\right)
\cdot\dot{\Delta}_{j} u_1 \\
&\quad+\sup\limits_{j\in\mathbb{Z}}2^{-2sj}\int  \dot{\Delta}_{j}g_3
\cdot\dot{\Delta}_{j} n_2+  \dot{\Delta}_{j}\left(g_4+u_1\times B\right)
\cdot\dot{\Delta}_{j} u_2 +2\dot{\Delta}_{j}g_5\cdot\dot{\Delta}_{j} E \\
&\lesssim\norm{g_1}_{\dot{B}_{2,\infty}^{-s}} \norm{n_1}_{\dot{B}_{2,\infty}^{-s}}+
\norm{ g_2+u_2\times B}_{\dot{B}_{2,\infty}^{-s}} \norm{u_1}_{\dot{B}_{2,\infty}^{-s}}\\
&\quad+\norm{g_3}_{\dot{B}_{2,\infty}^{-s}} \norm{n_2}_{\dot{B}_{2,\infty}^{-s}}+\norm{ g_4+u_1\times B}_{\dot{B}_{2,\infty}^{-s}} \norm{u_2}_{\dot{B}_{2,\infty}^{-s}}+\norm{g_5}_{\dot{B}_{2,\infty}^{-s}} \norm{E}_{\dot{B}_{2,\infty}^{-s}}.\nonumber
\end{split}
\end{equation}
Then the proof is exactly the same as the proof of Lemma \ref{negative Sobolev norm} except that we should apply Lemma \ref{Lp embedding} instead to estimate the $\dot{B}_{2,\infty}^{-s}$ norm. Note that we allow $s=3/2$.
\end{proof}

\section{Proof of theorems}\label{section3}

\subsection{Proof of Theorem \ref{existence}}

In this subsection, we will prove the unique global solution to the system \eqref{yi}, and the key point is that we only assume the $H^3$ norm of initial data is small.

{\it Step 1. Global small $\mathcal{E}_3$ solution.}

We first close the energy estimates at the $H^3$ level by assuming a priori that $ \sqrt{\mathcal{E}_3(t)}\le \delta$ is sufficiently small.
Taking $k=0,1$ in \eqref{energy 1} of Lemma \ref{energy lemma} and then summing up, we obtain
\begin{equation} \label{end 1}
\begin{split}
\frac{d}{dt}\sum_{l=0}^3 \norm{ \na^{l}U}_{L^2}^2 +\la\sum_{l=0}^3\norm{\na^{l}(u_1,u_2)}_{L^2}^2\lesssim  \sqrt{\mathcal{E}_3}\mathcal{D}_3+\sqrt{\mathcal{D}_3}\sqrt{\mathcal{D}_3}\sqrt{\mathcal{E}_3} \lesssim  \delta\mathcal{D}_3.\end{split}
\end{equation}
Taking $k=0,1$ in \eqref{other dissipation} of Lemma \ref{other di} and then summing up, we obtain
\begin{eqnarray}  \label{end 2}
&&\frac{d}{dt}\left(\sum_{l=0}^2\int \na^lu_1\cdot\na\na^{l} n_1+\na^lu_2\cdot\na\na^{l} n_2 +\sum_{l=0}^2\int \na^{l}u_2\cdot\na^lE  -\eta\sum_{l=0}^1\int
\na^l E \cdot\na^{l}\na\times B \right)\nonumber
\\&&\quad+\la \left(\sum_{l=1}^3\norm{\na^{l}n_1}_{L^2}^2+\sum_{l=0}^3\norm{\na^{l}n_2}_{L^2}^2+\sum_{l=0}^2\norm{\na^{l}E}_{L^2}^2+ \sum_{l=1}^2 \norm{\na^{l} B}_{L^2}^2\right)\nonumber
\\ &&\qquad\lesssim
 \sum_{l=0}^3\norm{\na^{l} (u_1,u_2)}_{L^2}^2
 +\delta^2\mathcal{D}_3 .
\end{eqnarray}
Since $\delta$ is small, we deduce from $\eqref{end 2}\times\varepsilon+\eqref{end 1}$ that there exists an instant energy functional $\widetilde{\mathcal{E}}_3$ equivalent to ${\mathcal{E}}_3$ such that
\begin{equation*}
\frac{d}{dt}\widetilde{\mathcal{E}}_3+\mathcal{D}_3\le 0.
\end{equation*}
Integrating the inequality above directly in time, we obtain \eqref{energy inequality}. By a standard continuity argument, we then close the a priori estimates if we assume at initial time that $\mathcal{E}_3(0)\le \delta_0$ is sufficiently small. This concludes the unique global small $\mathcal{E}_3$ solution.

{\it Step 2. Global $\mathcal{E}_N$ solution.}

We shall prove this by an induction on $N\ge 3$. By \eqref{energy inequality}, then \eqref{energy inequality N} is valid for $N=3$. Assume \eqref{energy inequality N} holds for $N-1$ (then now $N\ge 4$). Taking $k=0,\dots,N-2$ in \eqref{energy 1} of Lemma \ref{energy lemma} and then summing up, we obtain
\begin{eqnarray} \label{end 3}
&&\frac{d}{dt}\sum_{l=0}^{N} \norm{ \na^{l}U}_{L^2}^2 +\la\sum_{l=0}^{N}\norm{\na^{l} (u_1,u_2)}_{L^2}^2 \nonumber\\&& \quad\le C_N \sqrt{\mathcal{D}_{N-1}}\sqrt{\mathcal{E}_N}\sqrt{\mathcal{D}_N}+C\sqrt{\mathcal{D}_{N-1}}\sqrt{\mathcal{D}_N}\sqrt{\mathcal{E}_N}
\le C_N \sqrt{\mathcal{D}_{N-1}}\sqrt{\mathcal{E}_N}\sqrt{\mathcal{D}_N}.
\end{eqnarray}
Here we have used the fact that $3\le \frac{N-2}{2}+2\le N-2+1$ since $N\ge 4$.
Note that it is important that we have put the two first factors in \eqref{energy 1} into the dissipation.

Taking  $k=0,\dots,N-2$ in \eqref{other dissipation} of Lemma \ref{other di} and then summing up, we obtain
\begin{eqnarray}  \label{end 4}
&&\frac{d}{dt}\left(\sum_{l=0}^{N-1}\int \na^lu_1\cdot\na\na^{l} n_1+\na^lu_2\cdot\na\na^{l} n_2 +\sum_{l=0}^{N-1}\int \na^{l}u_2\cdot\na^lE -\eta\sum_{l=0}^{N-2}\int
\na^l E \cdot\na\times\na^{l} B \right)\nonumber
\\&&\quad+\la \left(\sum_{l=1}^{N}\norm{\na^{l}n_1}_{L^2}^2+\sum_{l=0}^{N}\norm{\na^{l}n_2}_{L^2}^2+\sum_{l=0}^{N-1}\norm{\na^{l}E}_{L^2}^2+ \sum_{l=1}^{N-1} \norm{\na^{l} B}_{L^2}^2\right)\nonumber
 \\ &&\qquad\le
 C\sum_{l=0}^{N}\norm{\na^{l} (u_1,u_2)}_{L^2}^2
 +C_N  \sqrt{\mathcal{D}_{N-1}}\sqrt{\mathcal{D}_N}\sqrt{\mathcal{E}_N}.
\end{eqnarray}
We deduce from $\eqref{end 4}\times\varepsilon+\eqref{end 3}$ that there exists an instant energy functional $\widetilde{\mathcal{E}}_N$ equivalent to $\mathcal{E}_N$ such that, by Cauchy's inequality,
\begin{equation*}
\frac{d}{dt}\widetilde{\mathcal{E}}_N+\mathcal{D}_N\le  C_N \sqrt{\mathcal{D}_{N-1}}\sqrt{\mathcal{E}_N}\sqrt{\mathcal{D}_N}
 \le \varepsilon \mathcal{D}_N+ C_{N,\varepsilon} {\mathcal{D}_{N-1}} {\mathcal{E}_N}.
\end{equation*}
This implies
\begin{equation*}
\frac{d}{dt}\widetilde{\mathcal{E}}_N+\frac{1}{2}\mathcal{D}_N
 \le  C_{N } {\mathcal{D}_{N-1}} {\mathcal{E}_N}.
\end{equation*}
We then use the standard Gronwall lemma and the induction hypothesis to deduce that
\begin{eqnarray*}
 \mathcal{E}_N(t)+\int_0^t\mathcal{D}_N(\tau)\,d\tau
&& \le  C\mathcal{E}_N(0)e^{C_{N }\int_0^t{\mathcal{D}_{N-1}} (\tau)\,d\tau}
\\&&\le  C\mathcal{E}_N(0)e^{C_{N }P_{N-1}\left(\mathcal{E}_{N-1}(0)\right)}\\&&\le  C\mathcal{E}_N(0)e^{C_{N }P_{N-1}\left(\mathcal{E}_{N}(0)\right)}\equiv P_N\left(\mathcal{E}_{N}(0)\right).
\end{eqnarray*}
This concludes the global $\mathcal{E}_N$ solution. The proof of Theorem \ref{existence} is completed.\hfill$\Box$

\subsection{Proof of Theorem \ref{decay}}
In this subsection, we will prove the various time decay rates of the unique global solution to the system \eqref{yi} obtained in Theorem \ref{existence}. Fix $N\ge 5$. We need to assume that $\mathcal{E}_N(0)\le \delta_0=\delta_0(N)$ is small. Then Theorem \ref{existence} implies that there exists a unique global $\mathcal{E}_N$ solution, and $\mathcal{E}_N(t)\le P_{N}\left(\mathcal{E}_N(0)\right) \le \delta_0$ is small for all time $t$. Since now our $\delta_0$ is relative small with respect to $N$, we just ignore the $N$ dependence of the constants in the energy estimates in the previous section.

{\it Step 1. Basic decay.}

For the convenience of presentations, we define a family of energy functionals and the corresponding dissipation rates with {\it minimum derivative counts} as
\begin{equation}\label{1111}
\mathcal{E}_{k}^{k+2}=\sum_{l=k}^{k+2}\norm{ \na^{l}U}_{L^2}^2
\end{equation}
and
\begin{equation}\label{2222}
\mathcal{D}_{k}^{k+2}=\sum_{l=k+1}^{k+2}\norm{\na^{l}n_1}_{L^2}^2+\sum_{l=k}^{k+2}\norm{\na^{l}(n_2,u_1,u_2)}_{L^2}^2+\sum_{l=k}^{k+1}\norm{\na^{l}E}_{L^2}^2+ \norm{\na^{k+1} B}_{L^2}^2.
\end{equation}

By Lemma \ref{energy lemma}, we have that for $k=0,\dots,N-2$,
\begin{eqnarray}  \label{end 5}
&&\frac{d}{dt}\sum_{l=k}^{k+2}\norm{ \na^{l}U}_{L^2}^2+\la\sum_{l=k}^{k+2}\norm{\na^{l} (u_1,u_2)}_{L^2}^2\nonumber \\
&&\quad\lesssim \sqrt{\delta_0} \mathcal{D}_k^{k+2} +\norm{(n_2,u_1,u_2)}_{L^\infty}\norm{\na^{k+2}(n_1,n_2,u_1,u_2)}_{L^2}
\norm{ \na^{k+2}( E,  B )}_{L^2}.
\end{eqnarray}
By Lemma \ref{other di}, we have that for $k=0,\dots,N-2$,
\begin{eqnarray} \label{end 6}
&&\frac{d}{dt}\left(\sum_{l=k}^{k+1}\int \na^lu_1\cdot\na\na^{l}n_1+\na^lu_2\cdot\na\na^{l}n_2 +\sum_{l=k}^{k+1}\int \na^{l}u_2\cdot\na^lE  -\eta\int
\na^k E \cdot\na^{k}\na\times B \right)\nonumber
\\&&\quad+\la \left(\sum_{l=k+1}^{k+2}\norm{\na^{l}n_1}_{L^2}^2+\sum_{l=k}^{k+2}\norm{\na^{l}n_2}_{L^2}^2+\sum_{l=k}^{k+1}\norm{\na^{l}E}_{L^2}^2+ \norm{\na^{k+1} B}_{L^2}^2\right)\nonumber
\\ &&\qquad\lesssim\sum_{l=k}^{k+2}\norm{\na^{l} (u_1,u_2)}_{L^2}^2
 +\delta_0 \sum_{l=k}^{k+2}\norm{\na^{l} (u_1, u_2)}_{L^2}^2.
\end{eqnarray}
Since $\delta_0$ is small, we deduce from $\eqref{end 6}\times\varepsilon+\eqref{end 5}$ that there exists an instant energy functional $\widetilde{\mathcal{E}}_k^{k+2}$ equivalent to $\mathcal{E}_k^{k+2}$ such that
\begin{equation}\label{energy}
\frac{d}{dt}\widetilde{\mathcal{E}}_k^{k+2}+\mathcal{D}_k^{k+2}\lesssim   \norm{(n_2,u_1,u_2)}_{L^\infty}\norm{\na^{k+2}(n_1,n_2,u_1,u_2)}_{L^2}
\norm{ \na^{k+2}( E,  B )}_{L^2}.
\end{equation}
Note that we can not absorb the right-hand side of \eqref{energy} by the dissipation $\mathcal{D}_k^{k+2}$ since it does not contain $\norm{\na^{k+2}( E, B )}_{L^2}^2$. We will distinct the arguments by the value of $k$. If $k=0$ or $k=1$, we bound $\norm{\na^{k+2}( E, B )}_{L^2}$ by the energy. Then we have that for $k=0,1$,
\begin{equation*}
\frac{d}{dt}\widetilde{\mathcal{E}}_k^{k+2}+\mathcal{D}_k^{k+2}\lesssim   \sqrt{\mathcal{D}_k^{k+2}}\sqrt{\mathcal{D}_k^{k+2}}\sqrt{{\mathcal{E}}_3}\lesssim \sqrt{\delta_0}\mathcal{D}_k^{k+2},
\end{equation*}
which implies
\begin{equation*}
\frac{d}{dt}\widetilde{\mathcal{E}}_k^{k+2}+\mathcal{D}_k^{k+2}\le 0.
\end{equation*}
If $k\ge 2$, we have to bound $\norm{\na^{k+2}( E, B )}_{L^2}$ in term of $\norm{\na^{k+1}( E, B )}_{L^2}$ since $\sqrt{\mathcal{D}_k^{k+2}}$ can not control $\norm{(n_2,u_1,u_2)}_{L^\infty}$. The key point is to use the regularity interpolation method developed in \cite{GW,SG06}. By Lemma \ref{A1}, we have
\begin{eqnarray}\label{kkk}
&&\norm{(n_2,u_1,u_2)}_{L^\infty}\norm{ \na^{k+2}(n_1,n_2,u_1,u_2)}_{L^2} \norm{ \na^{k+2}( E, B )}_{L^2}\nonumber\\&&\quad\lesssim\norm{(n_2,u_1,u_2)}_{L^2}^{1-\frac{3 }{2k}}\norm{\na^{k}(n_2,u_1,u_2)}_{L^2}^{\frac{3 }{2k}}\norm{\na^{k+2}(n_1,n_2,u_1,u_2)}_{L^2}\nonumber\\
&&\qquad\cdot\norm{\na^{k+1}( E, B )}_{L^2}^{1-\frac{3 }{2k}}
\norm{\na^{\al}( E, B )}_{L^2}^{\frac{3 }{2k}},
\end{eqnarray}
where $\alpha$ is defined by
\begin{equation*}
 k+2=(k+1)\times\left(1-\frac{3 }{2k}\right)+\alpha\times \frac{3 }{2k}
 \Longrightarrow \alpha=\frac{5}{3}k+1.
\end{equation*}
Hence, for $k\ge 2$, if $N\ge \frac{5}{3}k+1\Longleftrightarrow 2\le k\le \frac{3}{5}(N-1)$, then by \eqref{kkk}, we deduce from \eqref{energy} that
\begin{equation*}
\frac{d}{dt}\widetilde{\mathcal{E}}_k^{k+2}+\mathcal{D}_k^{k+2} \lesssim  \sqrt{{\mathcal{E}}_{N}} {\mathcal{D}_k^{k+2}}\lesssim \sqrt{\delta_0}\mathcal{D}_k^{k+2},
\end{equation*}
which allow us to arrive at that for any integer $k$ with $0\le k\le  \frac{3}{5}(N-1) $ (note that $N-2\ge \frac{3}{5}(N-1)\ge 2$ since $N\ge 5$), we have
\begin{equation}\label{k energy}
\frac{d}{dt}\widetilde{\mathcal{E}}_k^{k+2}+\mathcal{D}_k^{k+2} \le 0.
\end{equation}

We now begin to derive the decay rate from \eqref{k energy}. In fact, we
have proved \eqref{H-sbound} or \eqref{H-sbound Besov} in the similar fashion of \cite{TWW} by utilizing Lemma \ref{negative Sobolev norm} and \ref{negative Besov norm}. Using Lemma
\ref{1-sinte}, we have that for $s\ge 0$ and $k+s>0$,
\begin{equation*}
\norm{\na^k (n_1,B) }_{L^2} \le  \norm{(n_1,B)}_{\dot{H}^{-s}}^{\frac{1}{k+1+s}}\norm{\na^{k+1}(n_1,B)}_{L^2}^{\frac{k+s}{k+1+s}}
 \le  C_0\norm{\na^{k+1}(n_1,B)}_{L^2}^{\frac{k+s}{k+1+s}}.
\end{equation*}
Similarly, using Lemma
\ref{Besov interpolation}, we  have that for $s>0$ and $k+s>0$,
\begin{equation*}
\norm{\na^k (n_1,B) }_{L^2} \le  \norm{(n_1,B)}_{\dot{B}_{2,\infty}^{-s}}^{\frac{1}{k+1+s}}\norm{\na^{k+1}(n_1,B)}_{L^2}^{\frac{k+s}{k+1+s}}
 \le  C_0\norm{\na^{k+1}(n_1,B)}_{L^2}^{\frac{k+s}{k+1+s}}.
\end{equation*}
On the other hand, for $k+2<N$, we have
\begin{equation*}
\norm{ \na^{k+2}( E, B )}_{L^2}\le  \norm{\na^{k+1}( E, B )}_{L^2}^{ \frac{N-k-2}{N-k-1}}
\norm{\na^N( E, B )}_{L^2}^{\frac{1}{N-k-1}}\le  C_0\norm{\na^{k+1}( E, B )}_{L^2}^{ \frac{N-k-2}{N-k-1}}.
\end{equation*}
Then we deduce from \eqref{k energy} that
\begin{equation*}
\frac{d}{dt}\widetilde{\mathcal{E}}_k^{k+2}+\left\{\mathcal{E}_k^{k+2}\right\}^{1+\vartheta} \le 0,
\end{equation*}
where $\vartheta=\max\left\{\frac{1}{k+s},\frac{1}{N-k-2}\right\}$.
Solving this inequality directly, we obtain in particular that
\begin{equation}\label{nnn}
\mathcal{E}_k^{k+2}  (t) \le \left\{\left[\mathcal{E}_k^{k+2}(0)\right]^{-\vartheta}+\vartheta  t\right\}^{-  {1}/{\vartheta}}
\le C_0 (1+ t)^{- {1}/{\vartheta}}=C_0 (1+ t)^{-\min\left\{ {k+s}, {N-k-2}\right\}}.
\end{equation}
Notice that \eqref{nnn} holds also for $k+s=0$ or $k+2=N$. So, if we want to obtain the optimal decay rate of the whole solution for the spatial derivatives of order $k$, we only need to assume $N$ large enough (for fixed $k$ and $s$) so  that $k+s\le N-k-2$. Thus we should require that
\begin{equation}
N\ge \max\left\{k+2, \frac{5}{3}k+1, 2k+2+s\right\}= 2k+2+s.\nonumber
\end{equation}
This proves the optimal decay \eqref{basic decay}.

{\it Step 2. Further decay.}

We first prove \eqref{further decay1} and \eqref{further decay11}. First, noticing that $-\nu g={\rm div} E$, by \eqref{basic decay} and Lemma \ref{A2}, if $N\ge2k+4+s$, then
\begin{equation}\label{n further decay}
\norm{\na^kn_2(t)}_{L^2}\lesssim\norm{\na^kg(t)}_{L^2}\lesssim\norm{\na^{k+1}E(t)}_{L^2}\lesssim C_0(1+t)^{-\frac{k+1+s}{2}}.
\end{equation}

Next, applying $\na^k$ to $\eqref{yi}_2, \eqref{yi}_4, \eqref{yi}_5$ and then multiplying the resulting identities
by $\na ^ku_1,$ $\na^ku_2,$ $\na^kE$ respectively, summing up and integrating over $\mathbb{R}^3$, we obtain
\begin{eqnarray}\label{uE yi}
&&\frac{d}{dt}\int\left(\frac{1}{2}\norms{\na^k(u_1,u_2)}^2+\norms{\na^kE}^2\right)+\nu \norm{\na^k(u_1,u_2)}_{L^2}^2\nonumber\\
&&\quad=\int\na^k\left(-\na n_1+g_2+u_2\times B\right)\cdot\na^ku_1+\int\na^k\left(-\na n_2+g_4+u_1\times B\right)\cdot\na^ku_2\nonumber\\
&&\qquad+2\nu\int\na^k\left( \na\times B+ g_5\right)\cdot\na^kE\nonumber\\
&&\quad\lesssim \norm{\na^{k+1} n_1}_{L^2} \norm{\na^{k} u_1}_{L^2}  +\norm{\na^{k}\left(g_2+u_2\times B\right)}_{L^2}\norm{\na^{k} u_1}_{L^2}\nonumber\\
&&\qquad+\norm{\na^{k+1} n_2}_{L^2} \norm{\na^{k} u_2}_{L^2}  +\norm{\na^{k}\left(g_4+u_1\times B\right)}_{L^2}\norm{\na^{k} u_2}_{L^2}\nonumber\\
&&\qquad+\norm{ \na^{k}\left(\na\times B+g_5
\right)}_{L^2}\norm{\na^{k} E}_{L^2}.
\end{eqnarray}
On the other hand, taking $l=k$ in \eqref{Eyi}, we may have
\begin{eqnarray} \label{uE san}
-\int  \na^k \partial_tu_2 \cdot\na^{k}E +2\nu \norm{\na^{k} E}_{L^2}^2&&\le \int \na\na^{k} n_2\cdot \na^{k}E +C \norm{\na^{k}(u_1,u_2)}_{L^2}\norm{\na^{k}  E}_{L^2}\nonumber\\
&&\quad+\norm{\na^{k}\left(g_4+u_1\times B\right)}_{L^2}\norm{\na^{k } E}_{L^2}.
\end{eqnarray}
Substituting \eqref{Eer} with $l=k$ into \eqref{uE san}, we may then have
\begin{eqnarray}\label{uE wu}
-&&\frac{d}{dt}\int  \na^ku_2 \cdot\na^{k}E +2\nu \norm{\na^{k} E}_{L^2}^2\nonumber\\
 &&\quad\lesssim \norm{\na^{k} u_2}_{L^2}^2+\left(\norm{\na^{k+1} n_2}_{L^2}+ \norm{\na^{k}(u_1,u_2)}_{L^2}\right)\norm{\na^{k}  E}_{L^2}\nonumber\\
 &&\qquad+\norm{ \na^{k}\left(\na\times B+g_5
\right)}_{L^2}\norm{\na^{k} u_2}_{L^2} +\norm{\na^{k}\left(g_4+u_1\times B\right)}_{L^2}\norm{\na^{k } E}_{L^2}.
\end{eqnarray}
Since $\varepsilon$ is small, we deduce from $\eqref{uE wu}\times\varepsilon$$+\eqref{uE yi}$ that there exists $\mathcal{F}_k(t)\sim\norm{\na^k(u_1,u_2,E)(t)}_{L^2}^2$ such that, by Cauchy's inequality, Lemma \ref{commutator}, \eqref{Iyi kk+1}, \eqref{I3 k}, \eqref{I4 k k+1}, \eqref{basic decay} and \eqref{n further decay},
\begin{eqnarray}\label{uE liu}
 &&\frac{d}{dt}\mathcal{F}_k(t)+\mathcal{F}_k(t)\nonumber
 \\&&\quad\lesssim \norm{\na^{k+1} (n_1,n_2)}_{L^2}^2+\norm{\na^{k+1} B}_{L^2}^2 +\norm{\na^{k}\left(g_2+u_2\times B\right)}_{L^2}^2\nonumber\\
 &&\qquad+\norm{\na^{k}\left(g_4+u_1\times B\right)}_{L^2}^2 +\norm{ \na^{k}g_5}_{L^2}^2\nonumber
  \\&&\quad\lesssim \norm{\na^{k+1} (n_1,n_2, B)}_{L^2}^2 + \left(\norm{(u_1,u_2)}_{H^{\frac k2}}+\norm{\na B}_{L^2}\right)^2\norm{\na^{k+1}B}_{L^2}^2\nonumber\\
  &&\qquad+\norm{(n_1,n_2,u_1,u_2)}_{L^\infty}^2\norm{\na^{k+1}(n_1,n_2,u_1,u_2)}_{L^2}^2+\norm{\na n_1}_{L^\infty}^2\norm{\na^kn_1}_{L^2}^2\nonumber
\\
&&\quad\le C_0(1+t)^{-(k+1+s)},
\end{eqnarray}
where we required $N\ge2k+4+s$.
Applying the standard Gronwall lemma to \eqref{uE liu}, we obtain
\begin{equation}
\begin{split}
\mathcal{F}_k(t)\le \mathcal{F}_k(0)e^{-t}+C_0\int_0^te^{-(t-\tau)}(1+\tau)^{-(k+1+s)}\,d\tau\lesssim C_0(1+t)^{-(k+1+s)}.\nonumber
\end{split}
\end{equation}
This implies
\begin{equation*}
\norm{\na^k(u_1,u_2,E)(t)}_{L^2}\lesssim \sqrt{\mathcal{F}_k(t)}\lesssim C_0(1+t)^{-\frac{k+1+s}{2}}.
\end{equation*}
We thus complete the proof of \eqref{further decay1}. Notice that \eqref{further decay11} now follows by \eqref{n further decay} with the improved decay rate of $E$ in \eqref{further decay1}, just requiring $N\ge2k+6+s$.

Now we prove \eqref{further decay2}. Assuming $ B_\infty =0$, then we can extract the following system from  $\eqref{yi}_3$--$\eqref{yi}_4$, denoting $\psi={\rm div} u_2$,
\begin{equation}\label{npsi yi}
\left\{
\begin{array}{lll}
\displaystyle\partial_tn_2+\psi=g_3,   \\
\displaystyle\partial_t \psi+\nu  \psi-2\nu^2 n_2=-\Delta n_2-{\rm div}(g_4+u_1\times B)+2\nu^2\left(-g-n_2\right).
\end{array}
\right.
\end{equation}
Here $g$ is defined in \eqref{gf}.
Applying $\na^k$ to $\eqref{npsi yi}$ and then multiplying the resulting identities
by $2\nu^2\na^kn_2$, $\na^k\psi$, respectively, summing up and integrating over $\mathbb{R}^3$, we obtain
\begin{eqnarray}\label{npsi er}
\frac{d}{dt}&&\int \nu^2\norms{\na^k n_2 }^2+\frac{1}{2}\norms{\na^k \psi }^2+\nu\norm{\na^k\psi}_{L^2}^2\nonumber\\
&&=2\nu^2\int\na^kg_3 \na^kn_2-\int\na^k \Delta n_2 \na^k\psi\nonumber\\
&&\quad-\int\na^k\left[ {\rm div}(g_4+u_1\times B)-2\nu^2\left(-g-n_2\right)\right] \na^k\psi.
\end{eqnarray}
Applying $\na^k$ to $\eqref{npsi yi}_2$ and then multiplying by $-\na^kn_2$, as before integrating by parts over $t$ and $x$ variables and using the equation $\eqref{npsi yi}_1$, we may obtain
\begin{eqnarray}\label{npsi san}
&&-\frac{d}{dt}\int\na^k\psi \na^kn_2+2\nu^2\norm{\na^kn_2}_{L^2}^2
=\norm{\na^k\psi}_{L^2}^2+\nu\int\na^kn_2 \na^k\psi-\int\na^kg_3 \na^k\psi\nonumber\\
&&\qquad\quad+\int\na^k\left[\Delta n_2+{\rm div}(g_4+u_1\times B)-2\nu^2\left(-g-n_2\right)\right] \na^kn_2.
\end{eqnarray}

Since $\varepsilon$ is small, we deduce from $\eqref{npsi san}\times\varepsilon+\eqref{npsi er}$ that there exists $\mathcal{G}_k(t)\sim\norm{\na^k(n_2,\psi)}_{L^2}^2$ such that, by Cauchy's inequality,
\begin{eqnarray}\label{npsi si}
 \frac{d}{dt}\mathcal{G}_k(t)+\mathcal{G}_k(t)&&\lesssim  \norm{\na^{k+2}n_2}_{L^2}^2+\norm{\na^{k}g_3}_{L^2}^2+\norm{\na^{k+1}g_4}_{L^2}^2+\norm{\na^{k+1}(u_1\times B)}_{L^2}^2\nonumber\\
 &&\quad+\norm{\na^{k}\left(-g-n_2\right)}_{L^2}^2.
\end{eqnarray}
By Lemma \ref{A2} and Lemma \ref{commutator}, we obtain
\begin{equation*}
\begin{split}
\norm{\na^{k}\left(-g-n_2\right)}_{L^2}^2&\lesssim \norm{n_1}_{L^\infty}^2
\norm{\na^{k}n_2}_{L^2}^2+\norm{\na^k n_1}_{L^2}^2\norm{n_2}_{L^\infty}^2\\
&\lesssim\delta_0
\norm{\na^{k}n_2}_{L^2}^2+\norm{n_2}_{L^\infty}^2\norm{\na^k n_1}_{L^2}^2.
\end{split}
\end{equation*}
By Lemma \ref{commutator} and Cauchy's inequality, we obtain
\begin{equation*}
\begin{split}
\norm{\na^{k+1}(u_1\times B)}_{L^2}^2&=\norm{u_1\times \na^{k+1}B+\left[\na^{k+1}, u_1\right]\times B}_{L^2}^2\\
&\lesssim\norm{u_1\times \na^{k+1}B}_{L^2}^2+\norm{\left[\na^{k+1}, u_1\right]\times B}_{L^2}^2\\
&\lesssim\norm{u_1}_{L^\infty}^2\norm{\na^{k+1}B}_{L^2}^2+\norm{\na u_1}_{L^\infty}^2\norm{\na^{k}B}_{L^2}^2+\norm{\na^{k+1}u_1}_{L^2}^2\norm{B}_{L^\infty}^2.
\end{split}
\end{equation*}
The other nonlinear terms on the right-hand side of \eqref{npsi si} can be estimated similarly.
Hence, we deduce from \eqref{npsi si} that, by \eqref{basic decay}--\eqref{further decay11},
\begin{eqnarray}\label{npsi qi}
 &&\frac{d}{dt}\mathcal{G}_k(t)+\mathcal{G}_k(t)\nonumber
 \\&&\quad\lesssim \norm{\na^{k+2}n_2}_{L^2}^2+\norm{u_1}_{L^\infty}^2\norm{\na^{k+1}B}_{L^2}^2+\norm{\na u_1}_{L^\infty}^2\norm{\na^{k}B}_{L^2}^2+\norm{B}_{L^\infty}^2\norm{\na^{k+1}u_1}_{L^2}^2\nonumber\\
&&\qquad+\norm{n_2}_{L^\infty}^2\norm{\na^k n_1}_{L^2}^2+\norm{(n_1,n_2,u_1,u_2)}_{L^\infty}^2\norm{\na^{k+2}(n_1,n_2,u_1,u_2)}_{L^2}^2 \nonumber\\ &&\qquad+\norm{\na(n_1,n_2,u_1,u_2)}_{L^\infty}^2\norm{\na^{k+1}(n_1,n_2,u_1,u_2)}_{L^2}^2\nonumber
\\&&\quad\le C_0\left((1+t)^{-(k+3+s)}+  (1+t)^{-(k+7/2+2s)} + (1+t)^{-(k+11/2+2s)} \right)\nonumber
\\&&\quad\le C_0(1+t)^{-(k+3+s)},
\end{eqnarray}
where we required $N\ge2k+8+s$.
Applying the Gronwall lemma to \eqref{npsi qi} again, we obtain
\begin{equation*}
\mathcal{G}_k(t)\le \mathcal{G}_k(0)e^{-t}+C_0\int_0^te^{-(t-\tau)}(1+\tau)^{-(k+3+s)}\,d\tau\le  C_0(1+t)^{-(k+3+s)}.
\end{equation*}
This implies
\begin{equation}\label{npsi ba}
\norm{\na^k(n_2,\psi)(t)}_{L^2}\lesssim \sqrt{\mathcal{G}_k(t) }\le C_0 (1+t)^{-\frac{k+3+s}{2}}.
\end{equation}
If required $N\ge2k+12+s$, then by \eqref{npsi ba}, we have
\begin{equation*}
\norm{\na^{k+2}n_2(t)}_{L^2} \lesssim C_0(1+t)^{-\frac{k+5+s}{2}}.
\end{equation*}
Having obtained such faster decay, we can then improve  \eqref{npsi qi} to be
\begin{equation*}
 \frac{d}{dt}\mathcal{G}_k(t)+\mathcal{G}_k(t)
 \le C_0\left((1+t)^{-(k+5+s)}+  (1+t)^{-(k+7/2+2s)}  \right)
  \le C_0(1+t)^{-(k+7/2+2s)}.
\end{equation*}
Applying the Gronwall lemma again, we obtain
\begin{equation*}
\norm{\na^k(n_2,\psi)(t)}_{L^2}\lesssim \sqrt{\mathcal{G}_k(t) }\le C_0 (1+t)^{-(k/2+7/4+s)}.
\end{equation*}
We thus complete the proof of \eqref{further decay2}. The proof of Theorem \ref{decay} is completed.

\smallskip\smallskip

\noindent{\bf Acknowledgments.}
\smallskip

\noindent This research is supported by the National Natural Science Foundation of China--NSAF (No. 10976026) and the National Natural
Science Foundation of China (Grant No. 11271305).

\end{document}